\documentclass[a4paper,11pt]{amsart}
\title[A geometric interpretation of Krull dimensions of $\boldsymbol{T}$-algebras]{A geometric interpretation of Krull dimensions of $\boldsymbol{T}$-algebras}

\author{Song JuAe}
\address{Song JuAe, Department of Mathematics, Graduate School of Science, Kyoto University, Kitashirakawa Oiwake-cho, Sakyo-ku, Kyoto 606-8502, Japan.}
\email{song.juae.8m@kyoto-u.ac.jp}

\author{Yasuhito Nakajima}
\address{Yasuhito Nakajima, Independent Scholar}
\email{yasuhito.nakajima.mathematics@gmail.com}


\usepackage{amsmath,amssymb,amscd}
\usepackage{amsthm,mathtools}
\usepackage{mathrsfs}
\usepackage{colonequals}

\newtheorem{dfn}{Definition}[section]
\newtheorem{thm}[dfn]{Theorem}
\newtheorem{prop}[dfn]{Proposition}
\newtheorem{cor}[dfn]{Corollary}
\newtheorem{lemma}[dfn]{Lemma}
\newtheorem{rem}[dfn]{Remark}
\newtheorem{ex}[dfn]{Example}

\def\Gamma{\varGamma}
\def\t{\,{}^t\!}

\begin{document}

\begin{abstract}
We investigate Krull dimensions of semirings and semifields dealt in tropical geometry.
For a congruence $C$ on a tropical Laurent polynomial semiring $\boldsymbol{T}[X_1^{\pm}, \ldots, X_n^{\pm}]$, a finite subset $T$ of $C$ is called a finite congruence tropical basis of $C$ if the congruence variety $\boldsymbol{V}(T)$ associated with $T$ coincides with $\boldsymbol{V}(C)$.
For $C$ proper, we prove that the Krull dimension of the quotient semiring $\boldsymbol{T}[X_1^{\pm}, \ldots, X_n^{\pm}] / C$ coincides with the maximum of the dimension of $\boldsymbol{V}(C)$ as a polyhedral complex plus one and that of $\boldsymbol{V}(C_{\boldsymbol{B}})$ when both $C$ and $C_{\boldsymbol{B}}$ have finite congruence tropical bases, respectively.
Here $C_{\boldsymbol{B}}$ is the congruence on $\boldsymbol{T}[X_1^{\pm}, \ldots, X_n^{\pm}]$ generated by $\{ (f_{\boldsymbol{B}}, g_{\boldsymbol{B}}) \,|\, (f, g) \in C \}$ and $f_{\boldsymbol{B}}$ is defined as the tropical Laurent polynomial obtained from $f$ by replacing the coefficients of all non $-\infty$ terms of $f$ with the real number zero.
With this fact, we also show that rational function semifields of tropical curves that do not consist of only one point have Krull dimension two.
\end{abstract}

\maketitle

\section{Introduction}
	\label{section1}

Tropical geometry is an algebraic geometry over the tropical semifield $\boldsymbol{T} \colonequals (\boldsymbol{R} \cup \{ -\infty \}, \operatorname{max}, +)$.
By an operation $\operatorname{trop}$ called tropicalization, a $d$-dimensional irreducible subvariety of an algebraic torus over a valuation field $K$ corresponds to the support of a balanced, pure, $d$-dimensional, $\boldsymbol{R}$-rational polyhedral complex (see \cite[Theorem~3.3.5]{Maclagan=Sturmfels} for more details).
In \cite{Giansiracusa=Giansiracusa2}, Giansiracusa--Giansiracusa introduced a concept of scheme-theoretic tropicalization, which sends an ideal $I$ in a Laurent polynomial ring $K[X_1^{\pm}, \ldots, X_n^{\pm}]$ over $K$ to a congruence $E(I)$ on the tropical Laurent polynomial semiring $\boldsymbol{T}[X_1^{\pm}, \ldots, X_n^{\pm}]$ over $\boldsymbol{T}$, and proved that the congruence variety $\boldsymbol{V}(E(I))$ associated with $E(I)$ equals the image of the algebraic set $\boldsymbol{V}(I)$ associated with $I$ by $\operatorname{trop}$.
For finitely generated congruences on tropical polynomial semirings, Bertram--Easton (\cite{Bertram=Easton}) and Jo\'{o}--Mincheva (\cite{Joo=Mincheva2}) proved a tropical analogue of Nullstellensatz.
As mentioned above, there is much research on tropical polynomial semirings, tropical Laurent polynomial semirings, and congruences on them.

In this paper, we determine the Krull dimension of the quotient semiring by a proper congruence $C$ on a tropical Laurent polynomial semiring when $C$ satisfies some mild conditions:

\begin{thm}[Theorem~\ref{thm:main1-2}]
    \label{thm:main1-1}
Let $C$ be a proper congruence on $\boldsymbol{T}[X_1^{\pm}, \ldots, X_n^{\pm}]$.
If both $C$ and $C_{\boldsymbol{B}}$ have finite congruence tropical bases, respectively, then the equality
\begin{align*}
    \operatorname{dim} \boldsymbol{T}[ X_1^{\pm}, \ldots, X_n^{\pm}] / C = \operatorname{max}\{ \operatorname{dim} \boldsymbol{V}(C) + 1, \operatorname{dim}\boldsymbol{V}(C_{\boldsymbol{B}}) \}
\end{align*}
holds.
\end{thm}

Here a finite congruence tropical basis of $C$ is a finite subset of $C$ whose congruence variety coincides with $\boldsymbol{V}(C)$ and $C_{\boldsymbol{B}}$ is the congruence on $\boldsymbol{T}[X_1^{\pm}, \ldots, X_n^{\pm}]$ generated by $\{ (f_{\boldsymbol{B}}, g_{\boldsymbol{B}}) \,|\, (f, g) \in C \}$, where $f_{\boldsymbol{B}}$ is defined as the tropical Laurent polynomial obtained from $f$ by replacing the coefficients of all non $-\infty$ terms of $f$ with the real number zero.
For $C$ above, since $\boldsymbol{V}(C)$ is a finite union of $\boldsymbol{R}$-rational polyhedral sets, we can define its dimension $\operatorname{dim}\boldsymbol{V}(C)$ as the dimension of its polyhedral structure.
On the other hand, in \cite{Joo=Mincheva2}, Jo\'{o}--Mincheva defined prime congruences on additively idempotent semirings as an analogue of prime ideals and the Krull dimension $\operatorname{dim}S$ of an additively idempotent semiring $S$ as the maximum length of strict inclusions of prime congruences on $S$.
Also in the same paper, they proved that both the tropical polynomial semiring and the tropical Laurent polynomial semiring in $n$-variables have Krull dimension $n + 1$ (and with this fact, proved the Nullstellenstz mentioned above).
For a more general setting, the polynomial semiring and the Laurent polynomial semiring in $n$-variables over an additively idempotent semiring $A$ of a finite Krull dimension has Krull dimension $\operatorname{dim}A + n$ by \cite{Joo=Mincheva1}.
Mincheva in \cite{Mincheva} verified that the quotient semiring $\boldsymbol{T}[X_1^{\pm}, \ldots, X_n^{\pm}] / E(I)$ has the dimension of $\boldsymbol{V}(I)$ plus one as its Krull dimension.
Since not all the supports of $\boldsymbol{R}$-rational polyhedral complexes are given by tropicalization, Theorem~\ref{thm:main1-1} is a proper generalization of Mincheva's result.
Theorem~\ref{thm:main1-1} is a tropical analogue of the classical fact that the dimension of an affine algebraic set is equal to the Krull dimension of its affine coordinate ring over an algebraically closed field (\cite[Proposition~1.7 in Chapter~I]{Hartshorne}).

We also investigate algebraic properties of rational function semifields $\operatorname{Rat}(\Gamma)$ of (abstract) tropical curves $\Gamma$ and with Theorem~\ref{thm:main1-1}, give a tropical analogue of the fact that the function field of a Riemann surface has transcendental degree one:

\begin{thm}[Theorem~\ref{thm:main2-2}]
    \label{thm:main2-1}
Let $\Gamma$ be a tropical curve which is not one point.
Then the equality
\begin{align*}
    \operatorname{dim}\operatorname{Rat}(\Gamma) = 2
\end{align*}
holds.
\end{thm}

A tropical curve is the extended distance space obtained from a finite connected graph by identifying each edge with a closed segment of $\boldsymbol{T}$ (see Subsection~\ref{subsection2.8} for more details).
As an analogue of linear systems on algebraic curves, the concept of linear systems on tropical curves was introduced and studied by several researchers (cf.~\cite{Haase=Musiker=Yu}).
A rational function (other than $-\infty$) on a tropical curve $\Gamma$ is a continuous integral affine function with a finite number of pieces and the set $\operatorname{Rat}(\Gamma)$ of all rational functions on $\Gamma$ forms a semifield with natural operations and called the rational function semifield of $\Gamma$.
In \cite{JuAe3}, Song proved that $\operatorname{Rat}(\Gamma)$ is finitely generated as a semifield over $\boldsymbol{T}$.
Hence there exists a sujective $\boldsymbol{T}$-algebra homomorphism $\psi$ from a tropical rational function semifield $\overline{\boldsymbol{T}(X_1, \ldots, X_n)}$ to $\operatorname{Rat}(\Gamma)$ (\cite[Lemma~3.11]{JuAe6}), and thus $\operatorname{Rat}(\Gamma)$ is isomorphic to the quotient semifield $\overline{\boldsymbol{T}(X_1, \ldots, X_n)} / \operatorname{Ker}(\psi)$ by the kernel congruence $\operatorname{Ker}(\psi)$ of $\psi$.
As in the tropical Laurent polynomial semiring case, we can define congruence varieties associated with congruences on tropical rational function semifields.
By \cite[Proposition~3.12]{JuAe6}, (the natural compactification of) the congruence variety associated with $\operatorname{Ker}(\psi)$ (to be regarded as a tropical curve by lattice length) is isomorphic to $\Gamma$.
More generally, for a finitely generated semifield $S$ over $\boldsymbol{T}$, since we do not fix a finite generating set of $S$, we can obtain distinct congruence varieties from $S$.
Song also verified that these congruence varieties correspond to each other with invertible maps defined by finite numbers of tropical rational functions (\cite[Corollary~3.20]{JuAe6}).
In addition, by \cite[Corollary~3.19]{JuAe7}, Song characterized rational function semifields of tropical curves with surjective $\boldsymbol{T}$-algebra homomorphisms from tropical rational function semifields and geometric properties (i.e., dimensions as supports of polyhedral complexes, connectivity, the existence of distinct parallel rays) of congruence varieties associated with the kernel congruences.

In the current paper, we directly extract these geometric properties without passing through congruence varieties.
Theorem~\ref{thm:main2-1} is one of them, which corresponds to the dimension property above.
Connectivity is described in terms of pseudodirect product, i.e., in the setting in \cite[Corollary~3.19]{JuAe7}, a finitely generated semifield $S$ over $\boldsymbol{T}$ defines (an isomorphism class of) a connected congruence variety if and only if there exist no two finitely generated semifields $S_1, S_2$ over $\boldsymbol{T}$ whose pseudodirect product $S_1 \bowtie S_2$ is isomorphic to $S$.
It is a tropical analogue of the classical fact that the affine scheme of a (unital commutative) ring $A$ is connected if and only if there exist no two nonzero (unital commutative) rings $A_1, A_2$ whose direct product $A_1 \times A_2$ is isomorphic to $A$ (\cite[Exercise~2.19 in Chapter~II]{Hartshorne}).
The existence of distinct parallel rays (of a fixed $\boldsymbol{R}$-rational polyhedral structure) is translated into the terms of numbers of unbounded cells of ($\boldsymbol{R}$-rational polyhedral structures of) congruence varieties and their recession fans, which depend on only the $\boldsymbol{T}$-algebra structure of $S$.

The rest of this paper is organized as follows.
In Section~\ref{section2}, we prepare several concepts that we need such as polyhedral sets, congruences on semirings and Krull dimensions of additively idempotent semirings, tropicalizations, and tropical curves and rational functions on them.
Section~\ref{section3} is our main section, and we prove all assertions above.

\section*{Acknowledgements}
The authors thank their supervisor Masanori Kobayashi, Takaaki Ito and Yuki Tsutsui for their helpful comments and discussions.

\section{Preliminaries}
	\label{section2}

In this section, we recall several definitions that we need later.
We refer to \cite{Golan} (resp.~\cite{Maclagan=Sturmfels}) for an introduction to the theory of semirings (resp.~tropical geometry) and employ definitions in \cite{Maclagan=Sturmfels} and \cite{Mikhalkin=Rau} (resp.~\cite{Bertram=Easton} and \cite{Joo=Mincheva2}, \cite{JuAe3}) related to polyhedral geometry (resp.~semirings, tropical curves) with slight modifications if we need.

In this paper, the font $\boldsymbol{a}$ (resp.~$\vec{a}$) is used to stand for a column (resp.~row) vector and $\t A$ denotes the transpose of a matrix $A$.

\subsection{Polyhedral sets}
	\label{subsection2.1}

A \textit{polyhedral set} in $\boldsymbol{R}^n$ is the solution set of a system of a finite number of linear inequalities.
The polyhedral set $\{ \boldsymbol{x} \in \boldsymbol{R}^n \,|\, A \boldsymbol{x} \ge \boldsymbol{0} \}$ is called the \textit{recession cone} of the polyhedral set $P \colonequals \{ \boldsymbol{x} \in \boldsymbol{R}^n \,|\, A \boldsymbol{x} \ge \boldsymbol{b} \}$ defined by an $l \times n$-matrix $A$ and a vector $\boldsymbol{b} \in \boldsymbol{R}^l$ when $P$ is nonempty.
Here $\boldsymbol{0}$ stands for the zero vector.
A \textit{polyhedral complex} $X$ is a complex consisting of a finite number of polyhedral sets.
Each poluhedral set in $X$ is a \textit{cell}.
Its \textit{support} $|X|$ is the union of its polyhedral sets.
The \textit{recession fan} of $X$ is the union of all recession cones of polyhedral sets in $X$.
By definition, we can say that it is also the \textit{recession fan} $\operatorname{rec}(|X|)$ of $|X|$.
A finite union of polyhedral sets in $\boldsymbol{R}^n$ is the support of a polyhedral complex in $\boldsymbol{R}^n$, and vice versa (cf.~\cite[Proposition~4.1.1(a)]{Mikhalkin=Rau}).

For a nonempty polyhedral set $P$ in $\boldsymbol{R}^n$, its \textit{dimension} $\operatorname{dim}P$ is defined by the dimension of the smallest affine subspace of $\boldsymbol{R}^n$ containing $P$.
If $P$ is empty, we define that it has dimension $-\infty$.
The \textit{dimension} $\operatorname{dim}X$ of a polyhedral complex $X$ consisting of polyhedral sets $P_1, \ldots, P_m$ is defined as $\operatorname{max}\{ \operatorname{dim}P_i \,|\, i = 1, \ldots, m\}$.
We say that $|X|$ has dimension $\operatorname{dim}X$, i.e., $\operatorname{dim}(|X|) \colonequals \operatorname{dim}X$.
Since $X$ consists of only a finite number of polyhedral sets, $\operatorname{dim}(|X|)$ is independent of the choice of its \textit{polyhedral structures}, i.e., polyhedral complexes with $|X|$ as their supports.
Clearly $\operatorname{dim}X \ge \operatorname{dim}(\operatorname{rec}(|X|))$ holds.

Let $G$ be a subgroup of the abelian group $(\boldsymbol{R}, +)$.
A polyhedral set $P$ in $\boldsymbol{R}^n$ is \textit{$G$-rational} if $P$ is of the form $\{ \boldsymbol{x} \in \boldsymbol{R}^n \,|\, A\boldsymbol{x} \ge \boldsymbol{b} \}$ with some $l \times n$-matrix $A$ with entries in $\boldsymbol{Q}$ and some vector $\boldsymbol{b} \in G^l$.
A polyhedral complex is \textit{$G$-rational} if it consists only of $G$-rational polyhedral sets.

\subsection{Semirings, algebras and semifields}
	\label{subsection2.2}

In this paper, a \textit{semiring} $S$ is a commutative semiring with the absorbing neutral element $0_S$ for addition $+$ and the identity $1_S$ for multiplication $\cdot$, i.e., a nonempty set $S$ with two binary operations $+$ and $\cdot$ satisfying the following four conditions:

$(1)$ the pair $(S, +)$ is a commutative monoid with an identity $0_S$,

$(2)$ the pair $(S, \cdot)$ is a commutative monoid with an identity $1_S$,

$(3)$ the triple $(S, +, \cdot)$ has the distributive property, i.e., $(a + b) \cdot c = a \cdot c + b \cdot c$ holds for any $a,b, c \in S$, and

$(4)$ the zero element $0_S$ is absorbing, i.e., $0_S \cdot a = 0_S$ holds for any $a \in S$.

If every nonzero element of a semiring $S$ is multiplicatively invertible and $0_S \not= 1_S$, then $S$ is called a \textit{semifield}.

The set $\boldsymbol{T} \colonequals \boldsymbol{R} \cup \{ -\infty \}$ with two tropical operations:
\begin{align*}
a \oplus b \colonequals \operatorname{max}\{ a, b \} \quad	\text{and} \quad a \odot b \colonequals a + b,
\end{align*}
where $a, b \in \boldsymbol{T}$ and $a + b$ stands for the usual sum of $a$ and $b$, becomes a semifield.
Here, for any $a \in \boldsymbol{T}$, we handle $-\infty$ as follows:
\begin{align*}
a \oplus (-\infty) = (-\infty) \oplus a = a \quad \text{and} \quad a \odot (-\infty) = (-\infty) \odot a = -\infty.
\end{align*}
This triple $(\boldsymbol{T}, \odot, \oplus)$ is called the \textit{tropical semifield}.
The subset $\boldsymbol{B} \colonequals \{ 0, -\infty \}$ of $\boldsymbol{T}$ becomes a semifield with tropical operations of $\boldsymbol{T}$ and is called the \textit{boolean semifield}.

The \textit{tropical polynomials} are defined in the usual way and the set of all tropical polynomials in $n$-variables is denoted by $\boldsymbol{T}[X_1, \ldots, X_n]$.
It becomes a semiring with two tropical operations and is called the \textit{tropical polynomial semiring}.
Similarly, we define the \textit{tropical Laurent polynomials} and the \textit{tropical Laurent polynomial semiring} $\boldsymbol{T}[X_1^{\pm}, \ldots, X_n^{\pm}]$.
Throughout this paper, our usage of the symbol $c \odot \boldsymbol{X}^{\odot \boldsymbol{m}}$ with $c \in \boldsymbol{T}$ and $\boldsymbol{m} = \t (m_1, \ldots, m_n) \in \boldsymbol{Z}^n$ means the tropical monomial $c \odot X_1^{\odot m_1} \odot \cdots \odot X_n^{\odot m_n}$.

A map $\varphi \colon S_1 \to S_2$ between semirings is a \textit{semiring homomorphism} if for any $a, b \in S_1$,
\begin{align*}
\varphi(a + b) = \varphi(a) + \varphi(b), \	\varphi(a \cdot b) = \varphi(a) \cdot \varphi(b), \	\varphi(0) = 0, \	\text{and}\	\varphi(1) = 1.
\end{align*}
When $\varphi$ is a bijective semiring homomorphism, it is a \textit{semiring isomorphism}.
Then $S_1$ is \textit{isomorphic to} $S_2$ as a semiring. 
Given a semiring homomorphism $\varphi \colon S_1 \to S_2$, we call the pair $(S_2, \varphi)$ (for short, $S_2$) a \textit{$S_1$-algebra}.
In particular, if both $S_1$ and $S_2$ are semifields and $\varphi$ is injective, then $S_2$ is a \textit{semifield over $S_1$}.
Then $S_1$ is frequently identified with its image by $\varphi$ in $S_2$.
For a semiring $S_1$, a map $\psi \colon (S_2, \varphi) \to (S_2^{\prime}, \varphi^{\prime})$ between $S_1$-algebras is a \textit{$S_1$-algebra homomorphism} if $\psi$ is a semiring homomorphism and $\varphi^{\prime} = \psi \circ \varphi$.
When there is no confusion, we write $\psi \colon S_2 \to S_2^{\prime}$ simply.
A bijective $S_1$-algebra homomorphism $S_2 \to S_2^{\prime}$ is a \textit{$S_1$-algebra isomorphism}.
Then $S_2$ and $S_2^{\prime}$ are said to be \textit{isomorphic} and denoted by $S_2 \cong S_2^{\prime}$.

A semiring $S$ is a $\boldsymbol{B}$-algebra if and only if it is \textit{additively idempotent}, i.e., $a + a = a$ holds for any $a \in S$.
Both $\boldsymbol{T}$ and $\boldsymbol{T}[X_1^{\pm}, \ldots, X_n^{\pm}]$ are $\boldsymbol{B}$-algebras.
When $S$ is additively idempotent, we can define a partial order as $a \ge b$ if and only if $a + b = a$.
If this partial order is a total order, then $S$ is said to be \textit{totally ordered}.
A $\boldsymbol{B}$-algebra $S$ is said to be \textit{cancellative} if whenever $a \cdot b = a \cdot c$ for some $a, b, c \in S$, then either $a = 0_S$ or $b = c$.
If $S$ is cancellative, then we can define the semifield $Q(S)$ of fractions as within the case of integral domains.
In this case, the map $S \to Q(S); x \mapsto x/1_S$ becomes an injective $\boldsymbol{B}$-algebra homomorphism.

\subsection{Congruences}
	\label{subsection2.3}

A \textit{congruence} $E$ on a semiring $S$ is a subset of $S^2 = S \times S$ satisfying

$(1)$ $(a, a) \in E$ holds for any $a \in S$,

$(2)$ if $(a, b) \in E$, then $(b, a) \in E$,

$(3)$ if $(a, b) \in E$ and $(b, c) \in E$, then $(a, c) \in E$,

$(4)$ if $(a, b) \in E$ and $(c, d) \in E$, then $(a + c, b + d) \in E$, and

$(5)$ if $(a, b) \in E$ and $(c, d) \in E$, then $(a \cdot c, b \cdot d) \in E$.

The diagonal set of $S^2$ is denoted by $\Delta$ and called the \textit{trivial} congruence on $S$.
It is the unique smallest congruence on $S$.
The set $S^2$ is a congruence on $S$ called the \textit{improper} congruence on $S$.
Congruences other than the improper congruence are said to be \textit{proper}.
Quotients by congruences can be considered in the usual sense and the quotient semiring of $S$ by the congruence $E$ is denoted by $S / E$.
Then the natural surjection $\pi_E \colon S \twoheadrightarrow S / E$ is a semiring homomorphism.

The intersection of (possibly infinitely many) congruences is again a congruence.
For a subset $T$ of $S^2$, let $\langle T \rangle_S = \langle T \rangle$ be the smallest congruence on $S$ containing $T$, i.e., the intersection of all congruences on $S$ containing $T$.
A congruence $E$ on $S$ is \textit{finitely generated} if there exist a finite number of element $(a_1, b_1), \ldots, (a_n, b_n) \in E$ such that $\langle \{ (a_1, b_1), \ldots, (a_n, b_n) \} \rangle = E$.

For a semiring homomorphism $\psi \colon S_1 \to S_2$, the \textit{kernel congruence} $\operatorname{Ker}(\psi)$ of $\psi$ is the congruence $\{ (a, b) \in S_1^2 \,|\, \psi(a) = \psi(b) \}$ on $S_1$.
For semirings and congruences on them, the fundamental homomorphism theorem holds (\cite[Proposition 2.4.4]{Giansiracusa=Giansiracusa2}).
Then, for the above $\pi_E$, we have $\operatorname{Ker}(\pi_E) = E$.

For two congruences $E, F$ on a semiring $S$ such that $E \subset F$, the image $F / E$ of $F$ in $S / E \times S / E$ is a congruence on $S / E$ and $(S / E) / (F / E)$ is isomorphic to $S / E$ as a semiring (\cite[Proposition~2.4.4(3)]{Giansiracusa=Giansiracusa2}).
This yields a bijection between congruences on $S / E$ and congruences on $S$ containing $E$.

\begin{lemma}
    \label{lem:finite}
Let $S$ be a semiring.
For a subset $E$ of $S^2$, the equality
\begin{align*}
\langle E \rangle = \bigcup_{F \subset E :\,\#F < \infty} \langle F \rangle
\end{align*}
holds.
\end{lemma}

\begin{proof}
For any finite subset $F$ of $E$, clearly $\langle E \rangle \supset \langle F \rangle$ holds, and thus $\langle E \rangle \supset \bigcup_{F \subset E : \#F < \infty} \langle F \rangle$ holds.

Clearly $E \subset \bigcup_{F \subset E : \#F < \infty} \langle F \rangle$, hence $\langle E \rangle \subset \left\langle \bigcup_{F \subset E : \#F < \infty} \langle F \rangle \right\rangle$.
Since it is easy to check that $\bigcup_{F \subset E : \#F < \infty} \langle F \rangle$ is a congruence on $S$, the equality $\left\langle \bigcup_{F \subset E : \#F < \infty} \langle F \rangle \right\rangle = \bigcup_{F \subset E : \#F < \infty} \langle F \rangle$ holds.
\end{proof}

\begin{rem}
	\label{rem1}
\upshape{
Let $C$ be a congruence on $\boldsymbol{T}[X_1^{\pm}, \ldots, X_n^{\pm}]$.
Then since $f \oplus m = f$ and $m \oplus (-\infty) = m$ for any term $m$ of $f$, by \cite[Proposition~2.2(iii)]{Joo=Mincheva2}, we have $(f, m), (m, -\infty) \in C$ if $(f, -\infty) \in C$ with $f \not= -\infty$.
As $m$ is multiplicatively invertible in $\boldsymbol{T}[X_1^{\pm}, \ldots, X_n^{\pm}]$, this implies $(0, -\infty) \in C$, and hence $C$ is improper (cf.~\cite[the proof of Lemma~3.3]{JuAe7}).
Conversely, if $C$ is improper, cleary $(f, -\infty) \in C$ does not imply $f = -\infty$.
Thus $C$ is proper if and only if $(f, -\infty) \in C$ implies $f = -\infty$.
}
\end{rem}

\subsection{Prime congruences and Krull dimensions}
	\label{subsection2.4}

Let $S$ be a semiring.
For two pairs $\alpha = (\alpha_1, \alpha_2), \beta = (\beta_1, \beta_2) \in S^2$, we define
\begin{align*}
\alpha \rtimes \beta \colonequals (\alpha_1 \beta_1 + \alpha_2 \beta_2, \alpha_1 \beta_2 + \alpha_2 \beta_1).
\end{align*}
A congruence $E$ on $S$ is a \textit{prime congruence} if $E$ is proper and $\alpha \rtimes \beta \in E$ implies that $\alpha \in E$ or $\beta \in E$ holds.

For two congruence $E$ and $F$ on $S$, the congruence on $S$ generated by $\{ \alpha \rtimes \beta \,|\, \alpha \in E, \beta \in F \}$ is denoted by $E \rtimes F$.

For a $\boldsymbol{B}$-algebra $A$, we define its \textit{Krull dimension} as the maximum length of strict inclusions of prime congruences on $A$.
For instance, the Krull dimension of $\boldsymbol{B}$ is zero, and that of $\boldsymbol{T}$ is one.
Like this, the Krull dimension of a semifield over $\boldsymbol{B}$ is not always zero.
By \cite[Proposition~2.10]{Joo=Mincheva2}, prime congruences are characterized as follows:
let $E$ be a congruence on $A$.
Then $E$ is a prime congruence if and only if the quotient $\boldsymbol{B}$-algebra is cancellative and totally ordered.

Let $U$ be an $l \times (n + 1)$-matrix with entries in $\boldsymbol{R}$ with an integer $l$ such that $1 \le l \le n + 1$.
This matrix $U$ is \textit{t-admissible} if

$(1)$ its first column $\boldsymbol{u}_1$ is the zero vector $\boldsymbol{0} \in \boldsymbol{R}^l$ or the first nonzero entry of $\boldsymbol{u}_1$ is positive, and 

$(2)$ for any integer $i$ such that $1 \le i \le l$, its $i$-th row is \textit{irredundant} on $\boldsymbol{R} \times \boldsymbol{Z}^n$, i.e., there exist $c \in \boldsymbol{R}$ and $\boldsymbol{n} \in \boldsymbol{Z}^n$ such that the first nonzero entry of $U \begin{pmatrix} c \\ \boldsymbol{n} \end{pmatrix}$ is $i$-th entry.

When $U$ above is t-admissible, then the rank of $U$ is $l$ and $U$ defines $P(U)$ as the congruence on $\boldsymbol{T}[X_1^{\pm}, \ldots, X_n^{\pm}]$ generated by the set of pairs $\left(c_1 \odot \boldsymbol{X}^{\odot \boldsymbol{n}_1} \oplus c_2 \odot \boldsymbol{X}^{\odot \boldsymbol{n}_2}, c_1 \odot \boldsymbol{X}^{\odot \boldsymbol{n}_1} \right)$ such that $U \begin{pmatrix} c_1 - c_2 \\ \boldsymbol{n}_1 - \boldsymbol{n}_2 \end{pmatrix}$ is $\boldsymbol{0}$ or its first nonzero entry is positive.
Note that $P(U)$ is written as $P(U)_{\boldsymbol{T}}$ in \cite{Joo=Mincheva2}.
Then $\boldsymbol{T}[X_1^{\pm}, \ldots, X_n^{\pm}] / P(U)$ is a totally ordered $\boldsymbol{B}$-algebra by the definition of $P(U)$ and cancellative, and hence $P(U)$ is in fact a prime congruence on $\boldsymbol{T}[X_1^{\pm}, \ldots, X_n^{\pm}]$.
For $1 \le i \le l$, let $U(i)$ denote the $i \times (n + 1)$-matrix consisting of the first $i$ rows of $U$, and $U(0) \colonequals \begin{pmatrix} 0 & \cdots & 0 \end{pmatrix}$ the $1 \times (n + 1)$-zero matrix.
Note that $U(l)$ is $U$ itself.
By definition, $U(i)$ is t-admissible for any $1 \le i \le l$, and
\begin{align*}
    P(U) \subsetneqq P(U(l - 1)) \subsetneqq \cdots \subsetneqq P(U(1)) \subsetneqq P(U(0))
\end{align*}
hold, where $P(U(0))$ is defined by the maximum proper congruence on $\boldsymbol{T}[X_1^{\pm}, \ldots, X_n^{\pm}]$, i.e., $\boldsymbol{T}[X_1^{\pm}, \ldots, X_n^{\pm}] / P(U(0))$ is isomorphic to $\boldsymbol{B}$.
Note that $P(U(0))$ is a prime congruence and admits the definition of $P(U)$ above.

\begin{thm}[{\cite[Theorem~4.13(iii)]{Joo=Mincheva2}}]
    \label{thm:prime1}
In the above setting, every congruence of $\boldsymbol{T}[X_1^{\pm}, \ldots, X_n^{\pm}]$ containing $P(U)$ is of the form $P(U(i))$ for some $0 \le i \le l$.
In particular, $\operatorname{dim}\boldsymbol{T}[X_1^{\pm}, \ldots, X_n^{\pm}] / P(U)$ coincides with the rank $\operatorname{rank}U$ of $U$.
\end{thm}

\begin{thm}[{\cite[Theorem~4.14(i)]{Joo=Mincheva2}}]
    \label{thm:prime2}
For every prime congruence $P$ on $\boldsymbol{T}[X_1^{\pm}, \ldots, X_n^{\pm}]$, there exists a t-admissible matrix $U$ such that $P = P(U)$ or $P = P(U(0))$.
\end{thm}

\begin{rem}
	\label{rem2}
By Theorems~\ref{thm:prime1} and \ref{thm:prime2}, for a congruence $E$ on $\boldsymbol{T}[X_1^{\pm}, \ldots, X_n^{\pm}]$, to calculate the Krull dimension of $\boldsymbol{T}[X_1^{\pm}, \ldots, X_n^{\pm}] / E$, it is enough to find a t-admissible matrix $U$ of maximum rank such that $E \subset P(U)$.
\end{rem}

\begin{rem}
	\label{rem3}
\upshape{
Since the quotient semiring $\boldsymbol{T}[X_1^{\pm}, \ldots, X_n^{\pm}] / P$ by a prime congruence $P$ is totally ordered, every tropical Laurent polynomial $f$ has a (possibly not unique) term $m_f$ such that $(f, m_f) \in P$.
Note that $m_f$ is one of the maximum terms of $f$ with respect to the term ordering defined by $P$.
Hence the pair of tropical Laurent polynomials $(f, g)$ is in $P$ if and only if so is the pair $(m_f, m_g)$.
}
\end{rem}

\subsection{Tropical rational function semifields and congruence varieties}
	\label{subsection2.5}

If $\overline{E}$ denotes the set
\begin{align*}
    \left\{ (f, g) \in \boldsymbol{T}[X_1, \ldots, X_n]^2 \,\middle|\, \forall x \in \boldsymbol{T}^n, f(x) = g(x) \right\},
\end{align*}
then it is a congruence on $\boldsymbol{T}[X_1, \ldots, X_n]$ and the semiring $\overline{\boldsymbol{T}[X_1, \ldots, X_n]} \colonequals \boldsymbol{T}[X_1, \ldots, X_n] / \overline{E}$ is a cancellative $\boldsymbol{B}$-algebra by \cite[Theorem 1]{Bertram=Easton} and \cite[Proposition 5.5 and Theorem 4.14(v)]{Joo=Mincheva2}.
We call it the \textit{tropical polynomial function semiring}.
We call its semifield of fractions the \textit{tropical rational function semifield} and write it as $\overline{\boldsymbol{T}(X_1, \ldots, X_n)}$.
In what follows, by abuse of notation, the image of each $X_i$ in $\overline{\boldsymbol{T}(X_1, \ldots, X_n)}$ is again written by $X_i$.
For a subset $T$ of $\boldsymbol{T}[X_1^{\pm}, \ldots, X_n^{\pm}]^2$ or $\overline{\boldsymbol{T}(X_1, \ldots, X_n)}^2$, we define $\boldsymbol{V}(T) \colonequals \{ x \in \boldsymbol{R}^n \,|\, \forall (f, g) \in T, f(x) = g(x)\}$ and call it the \textit{congruence variety} associated with $T$.
Note here that $\boldsymbol{V}(T) = \boldsymbol{V}(\langle T \rangle)$ holds by definition.
For an element $(f, g)$ of $\boldsymbol{T}[X_1^{\pm}, \ldots, X_n^{\pm}]^2$ or $\overline{\boldsymbol{T}(X_1, \ldots, X_n)}^2$, we use the notation $\boldsymbol{V}((f, g))$ to indicate $\boldsymbol{V}(\{ (f, g) \})$.
By defining subsets of $\boldsymbol{R}^n$ to be closed if they are of the form of congruence varieties, a topology on $\boldsymbol{R}^n$ is determined and coincides with the Euclidean topology on $\boldsymbol{R}^n$ by \cite[Lemma~3.7.4]{Giansiracusa=Giansiracusa}.
For a congruence $E$ on $\overline{\boldsymbol{T}(X_1, \ldots, X_n)}$, if $\boldsymbol{V}(E)$ is nonempty, then $\overline{\boldsymbol{T}(X_1, \ldots, X_n)} /E$ is a semifield over $\boldsymbol{T}$ by \cite[Lemma~3.6]{JuAe6}.
For a subset $V$ of $\boldsymbol{R}^n$, let
\begin{align*}
\boldsymbol{E}_{\pm}(V) \colonequals \left\{ (f, g) \in \boldsymbol{T}[X_1^{\pm}, \ldots, X_n^{\pm}]^2 \,\middle|\, \forall x \in V, f(x) = g(x) \right\}
\end{align*}
and
\begin{align*}
\boldsymbol{E}(V) \colonequals \left\{ (f, g) \in \overline{\boldsymbol{T}(X_1, \ldots, X_n)}^2 \,\middle|\, \forall x \in V, f(x) = g(x) \right\}.
\end{align*}
Then these are congruences on $\boldsymbol{T}[X_1^{\pm}, \ldots, X_n^{\pm}]$ and $\overline{\boldsymbol{T}(X_1, \ldots, X_n)}$, respectively.

The $\boldsymbol{B}$-algebra homomorphism from $\boldsymbol{T}$ to $\boldsymbol{B}$ sending any real number to $0$ and $-\infty$ to $-\infty$ is naturally extended to the $\boldsymbol{B}$-algebra homomorphism from $\boldsymbol{T}[X_1^{\pm}, \ldots, X_n^{\pm}]$ (resp.~$\overline{\boldsymbol{T}(X_1, \ldots, X_n)}$) to $\boldsymbol{T}[X_1^{\pm}, \ldots, X_n^{\pm}]$ (resp.~$\overline{\boldsymbol{T}(X_1, \ldots, X_n)}$).
In this paper, the image of $f \in \boldsymbol{T}[X_1^{\pm}, \ldots, X_n^{\pm}]$ (resp.~$f \in \overline{\boldsymbol{T}(X_1, \ldots, X_n)}$) by this $\boldsymbol{B}$-algebra homomorphism is denoted by $f_{\boldsymbol{B}}$.

\begin{rem}
	\label{rem4}
\upshape{
Let $U \colonequals \begin{pmatrix} 1 & a_1 & \cdots & a_n \end{pmatrix}$ with $\boldsymbol{a} = \t(a_1, \ldots, a_n) \in \boldsymbol{R}^n$.
Then clearly $U$ is t-admissible and $P(U) = \boldsymbol{E}(\boldsymbol{a})$ holds.
Thus for any congruence $C$ contained in this $P(U)$, we have $\boldsymbol{V}(C) \supset \boldsymbol{V}(P(U)) = \boldsymbol{V}(\boldsymbol{E}(\boldsymbol{a})) = \{ \boldsymbol{a} \}$.
In particular, if $U^{\prime}$ is a t-admissible matrix such that $U^{\prime}(1) = U$, then $\boldsymbol{V}(P(U^{\prime})) \supset \boldsymbol{V}(P(U)) = \boldsymbol{V}(\boldsymbol{E}(\boldsymbol{a})) = \{ \boldsymbol{a} \}$.
}
\end{rem}

\subsection{Tropicalization}
	\label{subsection2.6}

Let $K$ be a field.
A map $v \colon K \to \boldsymbol{T}$ is a \textit{valuation} if

(1) $v(0_K) = - \infty$ and $v(1_K) = 0$ hold,

(2) $v(ab) = v(a) + v(b)$ holds for any $a, b \in K$, and

(3) $v(a + b) \le \operatorname{max}\{ v(a), v(b) \}$ holds for any $a,b \in K$.
Moreover, $v(a) \not= v(b)$ implies $v(a + b) = \operatorname{max}\{ v(a), v(b) \}$.

The pair $(K, v)$ (for short, $K$) is called a \textit{valuation field}.
For a valuation field $K$, we can define the map $\operatorname{trop}$ from the Laurent polynomial ring $K[X_1^{\pm}, \ldots, X_n^{\pm}]$ in $n$-variables with coefficiencts in $K$ to  $\boldsymbol{T}[X_1^{\pm}, \ldots, X_n^{\pm}]$ by the correspondence
\begin{align*}
    \sum_{\boldsymbol{n} \in A} c_{\boldsymbol{n}} \boldsymbol{X}^{\boldsymbol{n}} \mapsto \bigoplus_{\boldsymbol{n} \in A} v(c_{\boldsymbol{n}}) \odot \boldsymbol{X}^{\odot \boldsymbol{n}},
\end{align*}
where $A$ is a nonempty finite subset of $\boldsymbol{Z}^n$.
The map $\operatorname{trop}$ is called a \textit{tropicalization map} and the image of an ideal $I$ of $K[X_1^{\pm}, \ldots, X_n^{\pm}]$ by $\operatorname{trop}$ is called the \textit{tropicalization} of $I$.

Every tropical Laurent polynomial $F$ defines the \textit{tropical hypersurface}
\begin{align*}
    \boldsymbol{V}(F) \colonequals \{ x \in \boldsymbol{R}^n \,|\, F \text{ has at least two terms which achieve its maximum at } x \}.
\end{align*}
If $F$ is of degree one, then $\boldsymbol{V}(F)$ is called a \textit{tropical hyperplane}.
When $F$ is of degree one and homogeneous, it is a \textit{tropical linear form}.

If $I$ is an ideal in the Laurent polynomial ring over a valuation field $K$ in $n$-variables and defines an algebraic subvariety $X$ in the algebraic torus $(K^{\ast})^n$, where $K^{\ast} \colonequals K \setminus \{ 0_K \}$, then the intersection
\begin{align*}
    \bigcap_{f \in I} \boldsymbol{V}(\operatorname{trop}(f)) \subset \boldsymbol{R}^n
\end{align*}
is called the \textit{tropicalization} of $X$ and denoted by $\operatorname{trop}(X)$.
Then there exist a finite number of Laurent polynomials $f_1, \ldots, f_l$ in $I$ satisfying $\bigcap_{i = 1}^l \boldsymbol{V}(\operatorname{trop}(f_i)) = \operatorname{trop}(X)$ (\cite[Theorem~2.6.6]{Maclagan=Sturmfels}).
When $I$ is prime, it is the support of a $v(K^{\ast})$-rational polyhedral complex of dimension $\operatorname{dim}X$ (in fact, it has a more rich structure by \cite[Theorem~3.3.5]{Maclagan=Sturmfels}).

\subsection{Scheme-theoretic tropicalization}
	\label{subsection2.7}

For a subset $T$ of $\boldsymbol{T}[X_1^{\pm}, \ldots, X_n^{\pm}]$, the \textit{bend congruence} $\operatorname{Bend}(T)$ is defined by the congruence generated by all $(f, f_{\hat{m}})$ with any $f \in T$ and its any term $m$, where $f_{\hat{m}}$ denotes the tropical Laurent polynomial obtained from $f$ by removing $m$.
Then the intersection of the tropical hypersurfaces defined by all elements of $T$ coincides with the congruence variety $\boldsymbol{V}(\operatorname{Bend}(T))$.
Moreover, in \cite[Theorem~1.1]{Maclagan=Rincon}, it was proved that the tropicalization $\operatorname{trop}(I)$ of an ideal $I$ in $K[X_1^{\pm}, \ldots, X_n^{\pm}]$ determines $\operatorname{Bend}(\operatorname{trop}(I))$ and vice versa, where $K$ is a valuation field.
This $\operatorname{Bend}(\operatorname{trop}(I))$ is called the \textit{scheme-theoretic tropicalization} of $I$.

\subsection{Tropical curves and rational functions}
	\label{subsection2.8}

In this paper, a \textit{graph} means an unweighted, undirected, finite, connected nonempty multigraph that may have loops.
For a graph $G$, the set of edges is denoted by $E(G)$.
A vertex $v$ of $G$ is a \textit{leaf end} if $v$ is incident to only one edge and this edge is not loop.
A \textit{leaf edge} is an edge of $G$ incident to a leaf end.

A \textit{tropical curve} is the topological space associated with the pair $(G, l)$ of a graph $G$ and a function $l \colon E(G) \to {\boldsymbol{R}}_{>0} \cup \{\infty\}$, where $l$ can take the value $\infty$ only on leaf edges, by identifying each edge $e$ of $G$ with the closed interval $[0, l(e)]$.
The interval $[0, \infty]$ is the one-point compactification of the interval $[0, \infty)$.
We regard $[0, \infty]$ not just as a topological space but as an extended metric space.
The distance between $\infty$ and any other point is infinite.
When $l(e)=\infty$, the leaf end of $e$ must be identified with $\infty$.
If $E(G) = \{ e \}$ and $l(e)=\infty$, then we can identify either leaf ends of $e$ with $\infty$.
For a point $x$ of a tropical curve $\Gamma$, if $x$ is identified with $\infty$, then $x$ is called a \textit{point at infinity}.

Let $\Gamma$ be a tropical curve.
A map $f \colon \Gamma \to \boldsymbol{R} \cup \{ \pm \infty \}$ is a \textit{rational function} on $\Gamma$ if $f$ is a constant function of $-\infty$ or a continuous piecewise affine function with integer slopes, with a finite number of pieces, that can take the values $\pm \infty$ at only points at infinity.
Let $\operatorname{Rat}(\Gamma)$ denote the set of all rational functions on $\Gamma$.
For rational functions $f, g \in \operatorname{Rat}(\Gamma)$ and a point $x$ of $\Gamma$ that is not a point at infinity, we define
\begin{align*}
(f \oplus g) (x) \colonequals \operatorname{max}\{f(x), g(x)\} \quad \text{and} \quad (f \odot g) (x) \colonequals f(x) + g(x).
\end{align*}
We extend $f \oplus g$ and $f \odot g$ to points at infinity to be continuous on the whole of $\Gamma$.
Then both are rational functions on $\Gamma$.
Note that for any $f \in \operatorname{Rat}(\Gamma)$, we have
\begin{align*}
f \oplus (-\infty) = (-\infty) \oplus f = f
\end{align*}
and
\begin{align*}
f \odot (-\infty) = (-\infty) \odot f = -\infty.
\end{align*}
Then $\operatorname{Rat}(\Gamma)$ becomes a semifield with these two operations.
Also, $\operatorname{Rat}(\Gamma)$ becomes a $\boldsymbol{T}$-algebra and a semifield over $\boldsymbol{T}$ with the natural inclusion $\boldsymbol{T} \hookrightarrow \operatorname{Rat}(\Gamma)$.

A vector $\boldsymbol{b} \in \boldsymbol{R}^n$ is \textit{primitive} if all its components are integers and their greatest common divisor is one.
When $\boldsymbol{b}$ is primitive, for $\lambda \ge 0$, the \textit{lattice length} of $\lambda \boldsymbol{b}$ is defined as $\lambda$ (cf.~\cite[Subsection 2.1]{JuAe2}).

\section{Main results}
	\label{section3}

In this section, we prove all assertions in Section~\ref{section1}.

First, we give an easy proposition:

\begin{prop}
    \label{prop1}
The congruence variety $\boldsymbol{V}(P)$ associated with a prime congruence $P$ on $\boldsymbol{T}[X_1^{\pm}, \ldots, X_n^{\pm}]$ or $\overline{\boldsymbol{T}(X_1, \ldots, X_n)}$ is empty or consists of only one point.
\end{prop}

\begin{proof}
For a point $\boldsymbol{z} = \t( z_1, \cdots, z_n ) \in \boldsymbol{R}^n$, let
\begin{align*}
f_{\boldsymbol{z}} \colonequals \left( \bigoplus_{i = 1}^n \left( z_i^{\odot (-1)} \odot X_i \oplus \left( z_i^{\odot (-1)} \odot X_i \right)^{\odot (-1)} \right) \right)^{\odot (-1)} \in \overline{\boldsymbol{T}(X_1, \ldots, X_n)}.
\end{align*}
Note that $f_{\boldsymbol{z}}$ is a tropical rational function that attains the maximum value zero at and only at $\boldsymbol{z}$.
Assume that $\boldsymbol{V}(P)$ has two distinct points $\boldsymbol{x}, \boldsymbol{y} \in \boldsymbol{R}^n$.
Then there exists $\varepsilon > 0$ such that the set of points where at least one of $\varepsilon \odot f_{\boldsymbol{x}} \oplus 0$ and $\varepsilon \odot f_{\boldsymbol{y}} \oplus 0$ takes the value zero is $\boldsymbol{R}^n$. 
Since both $\varepsilon \odot f_{\boldsymbol{x}} \oplus 0$ and $\varepsilon \odot f_{\boldsymbol{y}} \oplus 0$ are tropical rational functions, there exists $g \in \overline{\boldsymbol{T}[X_1, \ldots, X_n]} \setminus \{ -\infty \}$ such that $h_{\boldsymbol{x}} \colonequals (\varepsilon \odot f_{\boldsymbol{x}} \oplus 0) \odot g, h_{\boldsymbol{y}} \colonequals (\varepsilon \odot f_{\boldsymbol{y}} \oplus 0) \odot g \in \overline{\boldsymbol{T}[X_1, \ldots, X_n]}$.
It is easy to check that $(h_{\boldsymbol{x}}, g), (h_{\boldsymbol{y}}, g) \not\in P$ but $(h_{\boldsymbol{x}}, g) \rtimes (h_{\boldsymbol{y}}, g) \in \Delta \subset P$ when $P$ is a prime congruence on $\overline{\boldsymbol{T}(X_1, \ldots, X_n)}$, which is a contradiction.
By choosing proper preimages of $h_{\boldsymbol{x}}, h_{\boldsymbol{y}}, g$ in $\boldsymbol{T}[X_1^{\pm}, \ldots, X_n^{\pm}]$, the same argument holds, and thus it is also a contradiction in the case that $P$ is a prime congruence on $\boldsymbol{T}[X_1^{\pm}, \ldots, X_n^{\pm}]$.
\end{proof}

By Proposition~\ref{prop1} and by deducing from the classical setting (cf.~\cite[Proposition~1.7 in Chapter~I]{Hartshorne}), we wonder that for a congruence $C$ on $\boldsymbol{T}[X_1^{\pm}, \ldots, X_n^{\pm}]$ or $\overline{\boldsymbol{T}(X_1, \ldots, X_n)}$ whose congruence variety $\boldsymbol{V}(C)$ is a finite union of $\boldsymbol{R}$-rational polyhedral sets, prime congruences on $\boldsymbol{T}[X_1^{\pm}, \ldots, X_n^{\pm}]$ or $\overline{\boldsymbol{T}(X_1, \ldots, X_n)}$ containing $C$ is hardly related to the dimension of $\boldsymbol{V}(C)$
However, Mincheva proved the following theorem, which implies that the Krull dimension of $\boldsymbol{T}[X_1^{\pm}, \ldots, X_n^{\pm}] / C$ may reflect the dimension $\boldsymbol{V}(C)$:

\begin{thm}[{\cite[Theorem~7.2.1]{Mincheva}}]
    \label{thm:Mincheva}
Let $K$ be a valuation field and $I$ an ideal in $K[X_1^{\pm}, \ldots, X_n^{\pm}]$ that defines a $d$-dimensional subvariety of the algebraic torus $(K^{\ast})^n$.
Then the equality
\begin{align*}
    \operatorname{dim}\boldsymbol{T}[X_1^{\pm}, \ldots, X_n^{\pm}] / \operatorname{Bend}(\operatorname{trop}(I)) = d + 1
\end{align*}
holds.
\end{thm}

In this paper, we generalize this theorem as follows:

\begin{thm}[Theorem~\ref{thm:main1-1}]
    \label{thm:main1-2}
Let $C$ be a proper congruence on $\boldsymbol{T}[X_1^{\pm}, \ldots, X_n^{\pm}]$ and $C_{\boldsymbol{B}}$ the congruence $\langle \{ (f_{\boldsymbol{B}}, g_{\boldsymbol{B}}) \,|\, (f, g) \in C \} \rangle$ on $\boldsymbol{T}[X_1^{\pm}, \ldots, X_n^{\pm}]$.
If both $C$ and $C_{\boldsymbol{B}}$ have finite congruence tropical bases, respectively, then the equality
\begin{align*}
    \operatorname{dim} \boldsymbol{T}[ X_1^{\pm}, \ldots, X_n^{\pm}] / C = \operatorname{max} \{ \operatorname{dim} \boldsymbol{V}(C) + 1, \operatorname{dim} \boldsymbol{V}(C_{\boldsymbol{B}}) \}
\end{align*}
holds.
\end{thm}

Here the definition of finite congruence tropical bases is given as follows:

\begin{dfn}
	\label{dfn:finite congruence tropical basis}
\upshape{
Let $C$ be a congruence on $\boldsymbol{T}[X_1^{\pm}, \ldots, X_n^{\pm}]$.
A finite subset
\begin{align*}
\{ (f_1, g_1), \ldots, (f_l, g_l) \}    
\end{align*}
of $C$ is a \textit{finite congruence tropical basis} of $C$ if the equality
\begin{align*}
\boldsymbol{V}(C) = \bigcap_{i = 1}^l \boldsymbol{V}((f_i, g_i))
\end{align*}
holds.
}
\end{dfn}

\begin{lemma}[cf.~{\cite[the proof of Proposition~3.4]{JuAe7}}]
	\label{lem1}
For a congruence on $\boldsymbol{T}[X_1^{\pm}, \ldots, X_n^{\pm} ]$, if it has a finite congruence tropical basis, then it has a finite congruence tropical basis consisting of only one element.
\end{lemma}

\begin{proof}
Let $C$ be a congruence on $\boldsymbol{T}[X_1^{\pm}, \ldots, X_n^{\pm} ]$ that has a finite congruence tropical basis $\{ (f_1, g_1), \ldots, (f_l, g_l) \}$.
Since $(f_i \oplus g_i, f_i) = (f_i \oplus g_i, f_i \oplus f_i), (f_i \oplus g_i, g_i) = (f_i \oplus g_i, g_i \oplus g_i) \in C$ for any $i$, the pair $((f_i \oplus g_i)^{\odot 2}, f_i \odot g_i) = ((f_i \oplus g_i) \odot (f_i \oplus g_i), f_i \odot g_i)$ also belongs to $C$.
Hence $\left( \bigodot_{i = 1}^l (f_i \oplus g_i)^{\odot 2}, \bigodot_{i = 1}^l (f_i \odot g_i) \right) \in C$ follows.
Therefore the inclusion
\begin{align*}
\boldsymbol{V}\left(\left( \bigodot_{i = 1}^l (f_i \oplus g_i)^{\odot 2}, \bigodot_{i = 1}^l (f_i \odot g_i) \right)\right) \supset \boldsymbol{V}(C)
\end{align*}
holds.
Assume $x \in \boldsymbol{V}\left(\left( \bigodot_{i = 1}^l (f_i \oplus g_i)^{\odot 2}, \bigodot_{i = 1}^l (f_i \odot g_i) \right)\right)$, i.e., $\bigodot_{i = 1}^l (f_i(x) \oplus g_i(x))^{\odot 2} = \bigodot_{i = 1}^l (f_i(x) \odot g_i(x))$.
This implies $(f_i(x) \oplus g_i(x))^{\odot 2} = f_i(x) \odot g_i(x)$ for any $i$ since $(f_i(x) \oplus g_i(x))^{\odot 2} = f_i(x)^{\odot 2} \oplus f_i(x) \odot g_i(x) \oplus g_i(x)^{\odot 2} \ge f_i(x) \odot g_i(x)$.
On the other hand, $(f_i(x) \oplus g_i(x))^{\odot 2}$ is $f_i(x)^{\odot 2}$ when $f_i(x) \ge g_i(x)$; $g_i(x)^{\odot 2}$ when $f_i(x) < g_i(x)$, and thus $f_i(x) = g_i(x)$ must hold for each $i$.
Hence $x$ is in $\bigcap_{i = 1}^l \boldsymbol{V}((f_i, g_i)) = \boldsymbol{V}(C)$.
In conclusion,
\begin{align*}
\boldsymbol{V}\left(\left( \bigodot_{i = 1}^l (f_i \oplus g_i)^{\odot 2}, \bigodot_{i = 1}^l (f_i \odot g_i) \right)\right) = \boldsymbol{V}(C)
\end{align*}
holds.
\end{proof}

\begin{rem}
    \label{rem5}
\upshape{
For a congruence $C$ on $\boldsymbol{T}[X_1^{\pm}, \ldots, X_n^{\pm} ]$, if it has a finite congruence tropical basis, then its congruence variety $\boldsymbol{V}(C)$ is a finite union of $\boldsymbol{R}$-rational polyhedral sets (cf.~\cite{Grigoriev}, \cite[Corollary 3.5]{JuAe7}).
Then, by Lemma~\ref{lem1}, there exists a finite congruence tropical basis $T$ consisting of only one element $(f, g)$ and the exponents of terms of $f$ and $g$ (resp.~the coefficients of terms of $f$ and $g$ other than $-\infty$) define a finite number of matrices $A_1, \ldots, A_k$ whose all entries are integers (resp.~ vectors $\boldsymbol{b}_1, \ldots, \boldsymbol{b}_k$) so that
\begin{align*}
    \boldsymbol{V}(C) = \bigcup_{i = 1}^k \{ \boldsymbol{x} \in \boldsymbol{R}^n \,|\, A_i \boldsymbol{x} \ge \boldsymbol{b}_i \}.
\end{align*}
If $T_{\boldsymbol{B}}$ denotes $\{(f_{\boldsymbol{B}}, g_{\boldsymbol{B}})\}$, then the equality
\begin{align*}
    \boldsymbol{V}(T_{\boldsymbol{B}}) = \bigcup_{i = 1}^k \{ \boldsymbol{x} \in \boldsymbol{R}^n \,|\, A_i \boldsymbol{x} \ge \boldsymbol{0} \}
\end{align*}
holds.
If $\{ \boldsymbol{x} \in \boldsymbol{R}^n \,|\, A_i \boldsymbol{x} \ge \boldsymbol{b}_i \}$ is nonempty, then $\{ \boldsymbol{x} \in \boldsymbol{R}^n \,|\, A_i \boldsymbol{x} \ge \boldsymbol{0} \}$ is its recession cone.
Since the emptiness of $\{ \boldsymbol{x} \in \boldsymbol{R}^n \,|\, A_i \boldsymbol{x} \ge \boldsymbol{b}_i \}$ does not imply that of $\{ \boldsymbol{x} \in \boldsymbol{R}^n \,|\, A_i \boldsymbol{x} \ge \boldsymbol{0} \}$, the congruence variety $\boldsymbol{V}(T_{\boldsymbol{B}})$ may not be the recession fan of $\boldsymbol{V}(C)$.
See also the following examples.
And note here that $\boldsymbol{V}((f^{\prime}_{\boldsymbol{B}}, g^{\prime}_{\boldsymbol{B}}))$ contains the zero vector $\boldsymbol{0} \in \boldsymbol{R}^n$, so it is nonempty whenever $(f^{\prime}, g^{\prime})$ is any pair of tropical Laurent polynomials such that one of $f$ and $g$ is $-\infty$ and the other is not $-\infty$.
Thus by Remark~\ref{rem1}, every proper congruence $C$ on $\boldsymbol{T}[X_1^{\pm}, \ldots, X_n^{\pm}]$, the congruence variety $\boldsymbol{V}(C_{\boldsymbol{B}})$ contains the origin $\boldsymbol{0}$, where $C_{\boldsymbol{B}} \colonequals \langle \{ (f^{\prime}_{\boldsymbol{B}}, g^{\prime}_{\boldsymbol{B}}) \,|\, (f^{\prime}, g^{\prime}) \in C \} \rangle$ again.
}
\end{rem}

The following lemma is clear by \cite[Corollary~2.12]{Bertram=Easton}:

\begin{lemma}
	\label{lem2}
Let $T$ be a finite subset of $\boldsymbol{T}[X_1^{\pm}, \ldots, X_n^{\pm}]^2$ and let $T_{\boldsymbol{B}}$ denote the set $\{ (f_{\boldsymbol{B}}, g_{\boldsymbol{B}}) \,|\, (f, g) \in T \}$.
If $C$ is the congruence $\langle T \rangle$ on $\boldsymbol{T}[X_1^{\pm}, \ldots, X_n^{\pm}]$ and $C_{\boldsymbol{B}}$ the congruence $\langle \{ (f_{\boldsymbol{B}}, g_{\boldsymbol{B}}) \,|\, (f, g) \in C \} \rangle$ on $\boldsymbol{T}[X_1^{\pm}, \ldots, X_n^{\pm}]$, then the equality $C_{\boldsymbol{B}} = \langle T_{\boldsymbol{B}} \rangle_{\boldsymbol{T}[X_1^{\pm}, \ldots, X_n^{\pm}]}$ holds.
The same thing is true in the case that $T$ is a finite subset of $\overline{\boldsymbol{T}(X_1, \ldots, X_n)}$.
\end{lemma}

\begin{ex}
    \label{ex1}
\upshape{
Let $T \colonequals \{ ( f \colonequals X_1 \oplus 3 \odot X_2^{\odot 2}, g \colonequals 2 \odot X_1^{\odot 2} \oplus X_2) \} \subset \boldsymbol{T}[X_1^{\pm}, X_2^{\pm}]^2$.
It is a finite congruence tropical basis of the congruence $\langle T \rangle$.
For $\boldsymbol{y} \in \boldsymbol{R}^2$,
\begin{align*}
    A_1 \colonequals \begin{pmatrix} 1 & -2 \\ 2 & -1 \\ 1-2 & 0 \\ 2-1 & 0 \end{pmatrix} = \begin{pmatrix} 1 & -2 \\ 2 & -1 \\ -1 & 0 \\ 1 & 0 \end{pmatrix} \text{ and } \boldsymbol{b}_1 \colonequals \begin{pmatrix} 3 \\ -2 \\ 2 \\ -2 \end{pmatrix},
\end{align*}
the $\boldsymbol{R}$-rational polyhedral set $\{ \boldsymbol{x} \in \boldsymbol{R}^2 \,|\, A_1 \boldsymbol{x} \ge \boldsymbol{b}_1 \}$ contains $\boldsymbol{y}$ if and only if $f = X_1$, $g = 2 \odot X_1^{\odot 2}$ and $f = g$ hold at $\boldsymbol{y}$.
In the same way, $f$ and $g$ define each three more matrices $A_2, A_3, A_4$ with entries in $\boldsymbol{Z}$ and vectors $\boldsymbol{b}_2, \boldsymbol{b}_3, \boldsymbol{b}_4$ such that
\begin{align*}
    \boldsymbol{V}(T) = \bigcup_{i = 1}^4 \{ \boldsymbol{x} \in \boldsymbol{R}^2 \,|\, A_i \boldsymbol{x} \ge \boldsymbol{b}_i \}.
\end{align*}
Note that each of $\{ \boldsymbol{x} \in \boldsymbol{R}^2 \,|\, A_i \boldsymbol{x} \ge \boldsymbol{b}_i \}$ is a ray.
Similarly, for $T_{\boldsymbol{B}} \colonequals \{ (f_{\boldsymbol{B}}, g_{\boldsymbol{B}}) \}$, the congruence variety $\boldsymbol{V}(T_{\boldsymbol{B}})$ is the recession fan of $\boldsymbol{V}(C)$, i.e., the finite union of the recession cones $\bigcup_{i = 1}^4 \{ \boldsymbol{x} \in \boldsymbol{R}^2 \,|\, A_i \boldsymbol{x} \ge \boldsymbol{0} \}$.
Thus we have $\operatorname{dim}\boldsymbol{V}(T) = \operatorname{dim}\boldsymbol{V}(T_{\boldsymbol{B}}) = 1$.
}
\end{ex}

\begin{ex}
	\label{ex2}
\upshape{
Let $T \colonequals \{ (f \colonequals (-1) \odot X \oplus X^{\odot (-1)} \oplus 0, g \colonequals 0) \} \subset \boldsymbol{T}[X^{\pm}]^2$.
It is easy to check that $\boldsymbol{V}(T) = [0, 1] \subset \boldsymbol{R}$ and $\boldsymbol{V}(T_{\boldsymbol{B}}) = \{ 0 \} \subset \boldsymbol{R}$ for $T_{\boldsymbol{B}} \colonequals \{ (f_{\boldsymbol{B}}, g_{\boldsymbol{B}}) \}$.
Hence we have $\operatorname{dim}\boldsymbol{V}(T) = 1$ and $\operatorname{dim}\boldsymbol{V}(T_{\boldsymbol{B}}) = 0$.
}
\end{ex}

\begin{ex}
	\label{ex3}
\upshape{
Since $\boldsymbol{V}((X_1 \odot X_2^{\odot (-1)} \oplus 0, 0)) = \{ \t(x, y, z) \in \boldsymbol{R}^3 \,|\, x \le y \}$ and $\boldsymbol{V}((X_1^{\odot (-1)} \odot X_2 \oplus 0, 0)) = \{ \t(x, y, z) \in \boldsymbol{R}^3 \,|\, x \ge y \}$, by \cite[Lemma~3.9]{JuAe7}, we have $\boldsymbol{V}((X_1 \odot X_2^{\odot (-1)} \oplus X_1^{\odot (-1)} \odot X_2 \oplus 0, 0)) = \boldsymbol{V}((X_1 \odot X_2^{\odot (-1)} \oplus 0, 0)) \cap \boldsymbol{V}((X_1^{\odot (-1)} \odot X_2 \oplus 0, 0)) = \{ \t(x, x, z) \in \boldsymbol{R}^3 \}$.
By the same argument, we have $\boldsymbol{V}((f_1, 0)) = \{ \t(x, x, x) \in \boldsymbol{R}^3 \,|\, x \ge 0 \}$ with
\begin{align*}
f_1 \colonequals&~ X_1 \odot X_2^{\odot (-1)} \oplus X_1^{\odot (-1)} \odot X_2\\&~\oplus X_2 \odot X_3^{\odot (-1)} \oplus X_2^{\odot (-1)} \odot X_3 \oplus X_1^{\odot (-1)} \oplus 0.
\end{align*}
Similarly $\boldsymbol{V}((f_2, 0)) = \{ \t(x, 0, z) \in \boldsymbol{R}^3 \,|\, x \le 0, z \le 0 \}$ and $\boldsymbol{V}((f_3, 0)) = \{ \t(0, y, z) \in \boldsymbol{R}^3 \,|\, y \le 0, z \le 0 \}$ hold for
\begin{align*}
f_2 \colonequals&~ X_1 \oplus X_2 \oplus X_3 \oplus X_2^{\odot (-1)} \oplus 0,\\
f_3 \colonequals&~ X_1 \oplus X_2 \oplus X_3 \oplus X_1^{\odot (-1)} \oplus 0.
\end{align*}
By \cite[Lemma~3.10]{JuAe7}, $\boldsymbol{V}\left( \left( \left( f_2^{\odot (-1)} \oplus f_3^{\odot (-1)}\right)^{\odot (-1)}, 0 \right)\right) = \boldsymbol{V}((f_2 \oplus f_3, f_2 \odot f_3)) = \boldsymbol{V}((f_2, 0)) \cup \boldsymbol{V}((f_3, 0))$ hold.
Hence we have
\begin{align*}
&~ \boldsymbol{V} \left( \left( \left( f_1^{\odot (-1)} \oplus f_2^{\odot (-1)} \oplus f_3^{\odot (-1)} \right)^{\odot (-1)}, 0 \right) \right)\\
=&~ \boldsymbol{V}((f_2 \odot f_3 \oplus f_1 \odot f_3 \oplus f_1 \odot f_2, f_1 \odot f_2 \odot f_3))\\
=&~ \bigcup_{i = 1}^3 \boldsymbol{V}((f_i, 0))
\end{align*}
by \cite[Lemma~3.10]{JuAe7} again.
Let $f \colonequals f_2 \odot f_3 \oplus f_1 \odot f_3 \oplus f_1 \odot f_2$ and $g \colonequals f_1 \odot f_2 \odot f_3$ and 
\begin{align*}
f^{\prime}(X_1, X_2, X_3) \colonequals&~ f(1 \odot X_1, 1 \odot X_2, 2 \odot X_3),\\
g^{\prime}(X_1, X_2, X_3) \colonequals&~ g(1 \odot X_1, 1 \odot X_2, 2 \odot X_3).
\end{align*}
Then $T \colonequals \{ (f, g), (f^{\prime}, g^{\prime}) \}$ is a finite congruence tropical basis of the congruence $\langle T \rangle$ and $\boldsymbol{V}(T) = \boldsymbol{V}((f, g)) \cap \boldsymbol{V}((f^{\prime}, g^{\prime})) = \{ (0, 0, -1) \}$ as $\boldsymbol{V}((f^{\prime}, g^{\prime}))$ is the translation of $\boldsymbol{V}((f, g))$ in the direction $\t (-1, -1, -2)$.
Since $f^{\prime}_{\boldsymbol{B}} = f_{\boldsymbol{B}} = f$ and $g^{\prime}_{\boldsymbol{B}} = g_{\boldsymbol{B}} = g$, we have $\boldsymbol{V}(T_{\boldsymbol{B}}) = \boldsymbol{V}((f_{\boldsymbol{B}}, g_{\boldsymbol{B}})) = \boldsymbol{V}((f, g))$, where $T_{\boldsymbol{B}} \colonequals \{ (f_{\boldsymbol{B}}, g_{\boldsymbol{B}}), (f^{\prime}_{\boldsymbol{B}}, g^{\prime}_{\boldsymbol{B}}) \}$.
Thus in this case $\operatorname{dim}\boldsymbol{V}(T) = 0$ and $\operatorname{dim}\boldsymbol{V}(T_{\boldsymbol{B}}) = 2$.
}
\end{ex}

Theorem~\ref{thm:main1-2} is one of our main theorems in this paper.
To prove this theorem, we prepare some lemmas.
The following two lemmas are clear:

\begin{lemma}
    \label{lem3}
Let $U$ be a t-admissible matrix.
For any $k$ and $\varepsilon > 0$, the matrix $U^{\prime}$ obtained from $U$ by multiplying its $k$-th row by $\varepsilon$ is t-admissible and the equality
\begin{align*}
P(U) = P(U^{\prime})
\end{align*}
holds.
\end{lemma}

\begin{lemma}
    \label{lemm4}
Let $U$ be a t-admissible matrix.
For any $k$ and $j > k$, the matrix $U^{\prime}$ obtained from $U$ by adding its $k$-th row multiplied by a real number to its $j$-th row is t-admissible and the equality
\begin{align*}
P(U) = P(U^{\prime})
\end{align*}
holds.
\end{lemma}

For a while, fix
\begin{align*}
U \colonequals \begin{pmatrix} 0 & \vec{w}_1 \\  \vdots & \vdots \\ 0 & \vec{w}_{k- 2} \\ 0 & \vec{w}_{k - 1} \\ 1 & \vec{w}_k \\ a_{k + 1} & \vec{w}_{k + 1} \\ \vdots & \vdots \\ a_l & \vec{w}_l
\end{pmatrix}
\quad \text{and} \quad
U_{\varepsilon} \colonequals \begin{pmatrix} 0 & \vec{w}_1 \\ \vdots & \vdots \\ 0 & \vec{w}_{k - 2} \\ \varepsilon & \vec{w}_{k - 1} + \varepsilon \vec{w}_k \\ 1 & \vec{w}_k  \\ a_{k + 1} & \vec{w}_{k + 1} \\ \vdots & \vdots \\ a_l & \vec{w}_l
\end{pmatrix}
\end{align*}
for $a_i \in \boldsymbol{R}$, $\t \vec{w}_i \in \boldsymbol{R}^n$ and $\varepsilon > 0$ and assume that $U$ is t-admissible.

\begin{lemma}
	\label{lem:admissible1}
In the above setting, the matrix $U_{\varepsilon}$ is t-admissible for any $\varepsilon > 0$.
\end{lemma}

\begin{proof}
Since $U$ is t-admissible, any row except for the $k$-th row of $U_{\varepsilon}$ is irredundant for any $\varepsilon > 0$ and there exist $c_1, c_2 \in \boldsymbol{R}$ and $\boldsymbol{n}_1, \boldsymbol{n}_2 \in \boldsymbol{Z}^n$ such that for any $1 \le i \le k - 2$, $1 \le j \le k - 1$,
\begin{align*}
&\begin{pmatrix} 0 & \vec{w}_i \end{pmatrix} \begin{pmatrix} c_1 \\ \boldsymbol{n}_1\end{pmatrix} = 0 \quad \text{and}
\quad \begin{pmatrix} 0 & \vec{w}_{k - 1} \end{pmatrix} \begin{pmatrix} c_1 \\ \boldsymbol{n}_1\end{pmatrix} \not= 0,\\
&\begin{pmatrix} 0 & \vec{w}_j \end{pmatrix}\begin{pmatrix} c_2 \\ \boldsymbol{n}_2 \end{pmatrix} = 0 \quad \text{and} \quad \begin{pmatrix} 1 & \vec{w}_k \end{pmatrix} \begin{pmatrix} c_2 \\ \boldsymbol{n}_2 \end{pmatrix} \not= 0
\end{align*}
hold.
For such any $c_1, c_2, \boldsymbol{n}_1, \boldsymbol{n}_2$ and any $1 \le i \le k - 2$,
\begin{align*}
    \begin{pmatrix} 0 & \vec{w}_i \end{pmatrix} \begin{pmatrix} c_1 + c_2 \\ \boldsymbol{n}_1 + \boldsymbol{n}_2 \end{pmatrix} = 0
\end{align*}
holds.
As the conditions above say nothing for $c_1$, taking $c_1$ properly makes
\begin{align*}
&~\begin{pmatrix} \varepsilon & \vec{w}_{k - 1} + \varepsilon \vec{w}_k \end{pmatrix} \begin{pmatrix} c_1 + c_2 \\ \boldsymbol{n}_1 + \boldsymbol{n}_2 \end{pmatrix}\\
=&~ \left( \begin{pmatrix} 0 & \vec{w}_{k - 1}\end{pmatrix} + \varepsilon \begin{pmatrix} 1 & \vec{w}_k \end{pmatrix} \right) \left( \begin{pmatrix} c_1 \\ \boldsymbol{n}_1 \end{pmatrix} + \begin{pmatrix} c_2 \\ \boldsymbol{n}_2 \end{pmatrix} \right)\\
=&~ \begin{pmatrix} 0 & \vec{w}_{k - 1} \end{pmatrix} \begin{pmatrix} c_1 \\ \boldsymbol{n}_1 \end{pmatrix} + \begin{pmatrix} 0 & \vec{w}_{k - 1} \end{pmatrix} \begin{pmatrix} c_2 \\ \boldsymbol{n}_2 \end{pmatrix} + \varepsilon \begin{pmatrix} 1 & \vec{w}_k \end{pmatrix} \begin{pmatrix} c_1 + c_2 \\ \boldsymbol{n}_1 + \boldsymbol{n}_2 \end{pmatrix}\\
=&~ \begin{pmatrix} 0 & \vec{w}_{k - 1} \end{pmatrix} \begin{pmatrix} c_1 \\ \boldsymbol{n}_1 \end{pmatrix} + \varepsilon \begin{pmatrix} 1 & \vec{w}_k \end{pmatrix} \begin{pmatrix} c_1 + c_2 \\ \boldsymbol{n}_1 + \boldsymbol{n}_2 \end{pmatrix}
\end{align*}
be zero, i.e., the value $c_1$ can be taken as
\begin{align*}
    c_1 \colonequals - \frac{\vec{w}_{k - 1} \boldsymbol{n}_1}{\varepsilon} - c_2 - \vec{w}_k (\boldsymbol{n}_1 + \boldsymbol{n}_2)
\end{align*}
for each $c_2, \boldsymbol{n}_1, \boldsymbol{n}_2$ and $\varepsilon$ above.
For such $c_1$, since $\begin{pmatrix} 0 & \vec{w}_{k - 1} \end{pmatrix} \begin{pmatrix} c_1 \\ \boldsymbol{n}_1 \end{pmatrix} \not= 0$, the value $\begin{pmatrix} 1 & \vec{w}_k \end{pmatrix} \begin{pmatrix} c_1 + c_2 \\ \boldsymbol{n}_1 + \boldsymbol{n}_2 \end{pmatrix}$ is never zero.
Hence $U_{\varepsilon}$ is t-admissible.
\end{proof}

\begin{lemma}
    \label{lem:admissible2}
If $(f, g)$ is in $P(U)$, then there exists $\varepsilon > 0$ such that $U_{\delta}$ is t-admissible and $(f, g)$ is in $P(U_{\delta})$ for any $0 < \delta \le \varepsilon$.
\end{lemma}

\begin{proof}
Remark~\ref{rem3} implies that if there exists $\varepsilon > 0$ such that for any $0 < \delta \le \varepsilon$, the sets of maximum terms of $f$ and $g$ with respect to the term ordering defined by $P(U)$ coincides with those of $f$ and $g$ with respect to the term ordering defined by $P(U_{\delta})$ respectively, then the assertion holds.
For any nonnegative integer $i$ at most $\operatorname{rank}(U) = l$, let $M_{U(i)}(f)$ be the set of maximum terms of $f$ with respect to the term ordering defined by $P(U(i))$.
In the same way, $M_{U_{\varepsilon}(i)}(f)$ is defined for any $\varepsilon > 0$.
By definition, $M_{U(i)}(f) = M_{U_{\varepsilon}(i)}(f)$ holds for any $0 \le i \le k-2$ and there exist natural inclusions
\begin{align*}
M_{U(0)}(f) \supset M_{U(1)}(f) \supset \cdots \supset M_{U}(f)
\end{align*}
and
\begin{align*}
M_{U_{\varepsilon}(0)}(f) \supset M_{U_{\varepsilon}(1)}(f) \supset \cdots \supset M_{U_{\varepsilon}}(f).
\end{align*}
If there exists $\varepsilon > 0$ such that for any $0 < \delta \le \varepsilon$,
\begin{align*}
    M_{U(k)}(f) = M_{U_{\delta}(k)}(f),
\end{align*}
then, since $U$ and $U_{\delta}$ have the same rows from rows $k$,
\begin{align*}
    M_{U}(f) = M_{U_{\delta}}(f)
\end{align*}
holds for any $0 < \delta \le \varepsilon$, and hence $f$ has the same maximum terms with respect to the term orderings defined by $P(U)$ and $P(U_{\delta})$, respectively.

Assume that $M_{U(k)}(f) \cap M_{U_{\varepsilon}(k)}(f) \not= \varnothing$ and let $c \odot \boldsymbol{X}^{\odot \boldsymbol{n}} \in M_{U(k)}(f) \cap M_{U_{\varepsilon}(k)}(f)$.
For any $b \odot \boldsymbol{X}^{\odot \boldsymbol{m}} \in M_{U(k)}(f)$, the equalities
\begin{align*}
    \begin{pmatrix} 0 & \vec{w}_{k - 1} \end{pmatrix} \begin{pmatrix} b \\ \boldsymbol{m} \end{pmatrix} &= \begin{pmatrix} 0 & \vec{w}_{k - 1} \end{pmatrix} \begin{pmatrix} c \\ \boldsymbol{n} \end{pmatrix},\\
    \begin{pmatrix} 1 & \vec{w}_k \end{pmatrix} \begin{pmatrix} b \\ \boldsymbol{m} \end{pmatrix} &=
    \begin{pmatrix} 1 & \vec{w}_k \end{pmatrix} \begin{pmatrix} c \\ \boldsymbol{n} \end{pmatrix}
\end{align*}
hold, and thus
\begin{align*}
    \begin{pmatrix} \varepsilon & \vec{w}_{k - 1} + \varepsilon \vec{w}_k \end{pmatrix} \begin{pmatrix} b \\ \boldsymbol{m} \end{pmatrix} = \begin{pmatrix} \varepsilon & \vec{w}_{k - 1} + \varepsilon \vec{w}_k \end{pmatrix} \begin{pmatrix} c \\ \boldsymbol{n} \end{pmatrix}
\end{align*}
is also true.
The tropical monomial $b \odot \boldsymbol{X}^{\odot \boldsymbol{m}}$ belongs to $M_{U_{\varepsilon}(k)}(f)$ since $c \odot \boldsymbol{X}^{\odot \boldsymbol{n}} \in M_{U(k)}(f) \cap M_{U_{\varepsilon}(k)}(f)$.
By the similar argument, if $a \odot \boldsymbol{X}^{\odot \boldsymbol{l}} \in M_{U_{\varepsilon}(k)}(f)$, then $a \odot \boldsymbol{X}^{\odot \boldsymbol{l}} \in M_{U(k)}(f)$.
Therefore $M_{U(k)}(f) \cap M_{U_{\varepsilon}(k)}(f) \not= \varnothing$ implies $M_{U(k)}(f) = M_{U_{\varepsilon}(k)}(f)$.

Assume that $M_{U(k)}(f) \cap M_{U_{\varepsilon}(k)}(f) = \varnothing$.
For any $b \odot \boldsymbol{X}^{\odot \boldsymbol{m}} \in M_{U(k)}(f)$ and $a \odot \boldsymbol{X}^{\odot \boldsymbol{l}} \in M_{U_{\varepsilon}(k)}(f)$, since $b \odot \boldsymbol{X}^{\odot \boldsymbol{m}} \not\in M_{U_{\varepsilon}(k)}(f)$, one of the following holds:
\begin{align}
    \begin{pmatrix} \varepsilon & \vec{w}_{k - 1} + \varepsilon \vec{w}_k \end{pmatrix} \begin{pmatrix} a \\ \boldsymbol{l} \end{pmatrix} > 
    \begin{pmatrix} \varepsilon & \vec{w}_{k - 1} + \varepsilon \vec{w}_k \end{pmatrix} \begin{pmatrix} b \\ \boldsymbol{m} \end{pmatrix} \text{ or } \label{condition1}
\end{align}
\begin{equation}
    \begin{split}
    \begin{pmatrix} \varepsilon & \vec{w}_{k - 1} + \varepsilon \vec{w}_k \end{pmatrix} \begin{pmatrix} a \\ \boldsymbol{l} \end{pmatrix} &= 
    \begin{pmatrix} \varepsilon & \vec{w}_{k - 1} + \varepsilon \vec{w}_k \end{pmatrix} \begin{pmatrix} b \\ \boldsymbol{m} \end{pmatrix} \text{ and} \\
    \begin{pmatrix} 1 & \vec{w}_k \end{pmatrix} \begin{pmatrix} a \\ \boldsymbol{l} \end{pmatrix} &>
    \begin{pmatrix} 1 & \vec{w}_k \end{pmatrix} \begin{pmatrix} b \\ \boldsymbol{m} \end{pmatrix}. \label{condition2}
    \end{split}
\end{equation}
Similarly, $a \odot \boldsymbol{X}^{\odot \boldsymbol{l}} \not\in M_{U(k)}(f)$ causes one of the following:
\begin{align}
    \begin{pmatrix} 0 & \vec{w}_{k - 1} \end{pmatrix} \begin{pmatrix} a \\ \boldsymbol{l} \end{pmatrix} < 
    \begin{pmatrix} 0 & \vec{w}_{k - 1} \end{pmatrix} \begin{pmatrix} b \\ \boldsymbol{m} \end{pmatrix} \text{ or } \label{condition3}
\end{align}
\begin{equation}
    \begin{split}
    \begin{pmatrix} 0 & \vec{w}_{k - 1} \end{pmatrix} \begin{pmatrix} a \\ \boldsymbol{l} \end{pmatrix} &= 
    \begin{pmatrix} 0 & \vec{w}_{k - 1} \end{pmatrix} \begin{pmatrix} b \\ \boldsymbol{m} \end{pmatrix} \text{ and}\\
    \begin{pmatrix} 1 & \vec{w}_k \end{pmatrix} \begin{pmatrix} a \\ \boldsymbol{l} \end{pmatrix} &<
    \begin{pmatrix} 1 & \vec{w}_k \end{pmatrix} \begin{pmatrix} b \\ \boldsymbol{m} \end{pmatrix}. \label{condition4}
    \end{split}
\end{equation}
By conditions, \eqref{condition4} cannot occur with \eqref{condition1} or \eqref{condition2} simultaneously.
Hence \eqref{condition3} must hold.
If both \eqref{condition1} and \eqref{condition3} are true, then $\begin{pmatrix} 1 & \vec{w}_k \end{pmatrix} \begin{pmatrix} a \\ \boldsymbol{l} \end{pmatrix} > \begin{pmatrix} 1 & \vec{w}_k \end{pmatrix} \begin{pmatrix} b \\ \boldsymbol{m} \end{pmatrix}$ and
\begin{align*}
    \varepsilon >
    \frac{\left\{ \begin{pmatrix} 0 & \vec{w}_{k - 1} \end{pmatrix} \begin{pmatrix} b \\ \boldsymbol{m} \end{pmatrix} -
    \begin{pmatrix} 0 & \vec{w}_{k - 1} \end{pmatrix} \begin{pmatrix} a \\ \boldsymbol{l} \end{pmatrix} \right\}}
    { \left\{ \begin{pmatrix} 1 & \vec{w}_k \end{pmatrix} \begin{pmatrix} a \\ \boldsymbol{l} \end{pmatrix} - \begin{pmatrix} 1 & \vec{w}_k \end{pmatrix} \begin{pmatrix} b \\ \boldsymbol{m} \end{pmatrix} \right\} } > 0
\end{align*}
follow.
When both \eqref{condition2} and \eqref{condition3} hold, the following
\begin{align*}
    \varepsilon =
    \frac{\left\{ \begin{pmatrix} 0 & \vec{w}_{k - 1} \end{pmatrix} \begin{pmatrix} b \\ \boldsymbol{m} \end{pmatrix} -
    \begin{pmatrix} 0 & \vec{w}_{k - 1} \end{pmatrix} \begin{pmatrix} a \\ \boldsymbol{l} \end{pmatrix} \right\}}
    { \left\{ \begin{pmatrix} 1 & \vec{w}_k \end{pmatrix} \begin{pmatrix} a \\ \boldsymbol{l} \end{pmatrix} - \begin{pmatrix} 1 & \vec{w}_k \end{pmatrix} \begin{pmatrix} b \\ \boldsymbol{m} \end{pmatrix} \right\} } > 0
\end{align*}
hold.
Both cases are dependent on only $U$ and $f$.
In conclusion, $M_{U(k)}(f) = M_{U_{\varepsilon}(k)}(f)$, and hence $M_U (f) = M_{U_{\varepsilon}} (f)$ hold for a sufficientely small $\varepsilon > 0$.
The same argument for $g$ gives $\varepsilon > 0$ such that $(f, g) \in P(U_{\varepsilon})$, and for such $\varepsilon >0$ and any $0 < \delta \le \varepsilon$, the element $(f, g)$ is contained in $P(U_{\delta})$.
\end{proof}

Since $\varepsilon > 0$ can be retaken finitely many times, Lemmas~\ref{lem3}, \ref{lem:admissible1} and \ref{lem:admissible2} induce the following two corollaries:

\begin{cor}
	\label{cor1}
If $(f, g)$ is in $P(U)$, then there exists a t-admissible matrix $U^{\prime}$ whose $(1, 1)$-entry is one and whose rank is $l$ and that satisfies $(f, g) \in P(U^{\prime})$.
\end{cor}

\begin{cor}
    \label{cor2}
If $C$ is a finitely generated congruence on $\boldsymbol{T}[X_1^{\pm}, \ldots, X_n^{\pm}]$ contained in $P(U)$, then there exists a t-admissible matrix $U^{\prime}$ whose $(1, 1)$-entry is one and whose rank is $l$ and that satisfies $C \subset P(U^{\prime})$.
\end{cor}

Let $\boldsymbol{e}_1$ denote the vector $\t(1, 0, \ldots, 0) \in \boldsymbol{R}^n$.

\begin{prop}
    \label{prop2}
Let $C$ be a congruence on $\boldsymbol{T}[X_1^{\pm}, \ldots, X_n^{\pm}]$.
If $\boldsymbol{V}(C)$ contains a $d$-dimensional nonempty $\boldsymbol{R}$-rational polyhedral set, then there exists a t-admissible matrix $U$ whose rank is $d + 1$ and whose first column is $\boldsymbol{e}_1$ and that satisfies $C \subset P(U)$.
In particular, the inequalities
\begin{align*}
\operatorname{dim}\boldsymbol{V}(C) + 1 &\le \operatorname{max} \left\{ \operatorname{rank}U \,\middle|\, \begin{array}{l} U \text{ is a t-admissible matrix such that}\\ \text{the first column of } U \text{ is } \boldsymbol{e}_1 \\ \text{and } C \subset P(U) \end{array} \right\}\\
&\le \operatorname{max} \left\{ \operatorname{rank}U \,\middle|\, \begin{array}{l} U \text{ is a t-admissible matrix such that}\\ \text{the first column of } U \text{ is a nonzero vector} \\ \text{and } C \subset P(U) \end{array} \right\}
\end{align*}
and
\begin{align*}
    \operatorname{dim}\boldsymbol{T}[X_1^{\pm}, \ldots, X_n^{\pm}] / C \ge d + 1
\end{align*}
hold.
\end{prop}

\begin{proof}
Let $Y$ be a $d$-dimensional nonempty $\boldsymbol{R}$-rational polyhedral set contained in $\boldsymbol{V}(C)$.
Then there exist $\boldsymbol{w} \in Y$, linearly independent $\boldsymbol{u}_1, \ldots, \boldsymbol{u}_d \in \boldsymbol{Z}^n$ and $\varepsilon_1, \ldots, \varepsilon_d > 0$ satisfying $\boldsymbol{w} + \sum_{i = 1}^d \delta_i \boldsymbol{u}_i \in Y$ for any $i$ and $0 \le \delta_i \le \varepsilon_i$.
Let $\vec{w} \colonequals \t \boldsymbol{w}$, $\vec{u}_i \colonequals \t \boldsymbol{u}_i$ and
\begin{align*}
    U \colonequals
    \begin{pmatrix}
    1 & \vec{w} \\
    0 & \vec{u}_1\\
    \vdots & \vdots\\
    0 & \vec{u}_d
    \end{pmatrix}.
\end{align*}
Since $\boldsymbol{u}_1, \ldots, \boldsymbol{u}_d$ are linearly independent, $\operatorname{rank} U = d + 1$ holds and  for any $0 < k \le d$, there exists $\boldsymbol{n}_k \in \boldsymbol{Z}^n$ such that $\vec{u}_i \boldsymbol{n}_k = 0$ for any $1 \le i < k$ and $\vec{u}_k \boldsymbol{n}_k \not= 0$.
Thus such $\boldsymbol{n}_k$ satisfies 
\begin{align*}
    \begin{pmatrix} 1 & \vec{w} \end{pmatrix} \begin{pmatrix} -\vec{w} \boldsymbol{n}_k \\ \boldsymbol{n}_k \end{pmatrix} = 0
\end{align*}
and, for any $1 \le i < k$,
\begin{align*}
    \begin{pmatrix} 0 & \vec{u}_i \end{pmatrix} \begin{pmatrix} -\vec{w} \boldsymbol{n}_k \\ \boldsymbol{n}_k \end{pmatrix} = 0
\end{align*}
and
\begin{align*}
    \begin{pmatrix} 0 & \vec{u}_k \end{pmatrix} \begin{pmatrix} -\vec{w} \boldsymbol{n}_k \\ \boldsymbol{n}_k \end{pmatrix} \not= 0.
\end{align*}
This means that $U$ is t-admissible.

Next the inclusion $C \subset P(U)$ will be shown.
For $0 < \delta_i \le \varepsilon_i$, let
\begin{align*}
    U_{(\delta_1, \ldots, \delta_d)} \colonequals
    \begin{pmatrix*}[l]
    1 & \vec{w} \\
    1 & \vec{w} + \delta_1 \vec{u}_1\\
    1 & \vec{w} + \delta_1 \vec{u}_1 + \delta_2 \vec{u}_2\\
    \,\vdots & \,\vdots\\
    1 & \vec{w} + \delta_1 \vec{u}_1 + \cdots + \delta_d \vec{u}_d
    \end{pmatrix*}.
\end{align*}
Since $U$ is t-admissible, so is $U_{(\delta_1, \ldots, \delta_d)}$ and $P(U) = P(U_{(\delta_1, \ldots, \delta_d)})$ holds by Lemmas~\ref{lem3} and \ref{lemm4}.
Let $(f, g) \in C$.
Then $\boldsymbol{V}((f, g)) \supset \boldsymbol{V}(C) \supset Y$ hold.
Thus $f(\boldsymbol{w}) = g(\boldsymbol{w})$.
Since each term of $f$ is continuous, $\delta_1, \ldots, \delta_d$ can be chosen so that there exists a term $c \odot \boldsymbol{X}^{\odot \boldsymbol{n}}$ of $f$ such that $f(\boldsymbol{w}) = c \odot \boldsymbol{w}^{\odot \boldsymbol{n}} = \begin{pmatrix} 1 & \vec{w} \end{pmatrix} \begin{pmatrix} c \\ \boldsymbol{n} \end{pmatrix}$ and
\begin{align*}
f \left(\boldsymbol{w} + \sum_{i = 1}^j \delta_i \boldsymbol{u}_i \right) &= c \odot \left(\boldsymbol{w} + \sum_{i = 1}^j \delta_i \boldsymbol{u}_i \right)^{\odot \boldsymbol{n}} = \begin{pmatrix} 1 & \vec{w} + \sum_{i = 1}^j \delta_i \vec{u}_i \end{pmatrix} \begin{pmatrix} c \\ \boldsymbol{n} \end{pmatrix}
\end{align*}
for any $1 \le j \le d$.
Note that this term $c \odot \boldsymbol{X}^{\odot \boldsymbol{n}}$ is one of the maximum terms of $f$ with respect to the term ordering defined by $P(U_{(\delta_1, \ldots, \delta_d)})$.
By Remark~\ref{rem3}, the same argument for $g$ gives proper $\delta_1, \ldots, \delta_d$ ensuring $(f, g) \in P(U_{(\delta_1, \ldots, \delta_d)})$.
Since $P(U) = P(U_{(\delta_1, \ldots, \delta_d)})$, this concludes that $C \subset P(U)$ holds.
Hence
\begin{align*}
    \operatorname{dim} \boldsymbol{T}[X_1^{\pm}, \ldots, X_n^{\pm}] / C \ge \operatorname{dim} \boldsymbol{T}[X_1^{\pm}, \ldots, X_n^{\pm}] / P(U) = d + 1
\end{align*}
hold.
\end{proof}

\begin{prop}
	\label{prop3}
Let $C$ be a congruence on $\boldsymbol{T}[X_1^{\pm}, \ldots, X_n^{\pm}]$ having a finite congruence tropical basis.
If $U$ is a t-admissible matrix $U$ whose first column is a nonzero vector and satisfying $C \subset P(U)$, then $\boldsymbol{V}(C)$ contains a $(\operatorname{rank}U - 1)$-dimensional $\boldsymbol{R}$-rational polyhedral set.
In particular, if $\boldsymbol{V}(C)$ is empty, then there exists no such $U$, and the inequality
\begin{align*}
\operatorname{dim}\boldsymbol{V}(C) + 1 &\ge \operatorname{max} \left\{ \operatorname{rank}U \,\middle|\, \begin{array}{l} U \text{ is a t-admissible matrix such that}\\ \text{the first column of } U \text{ is a nonzero vector} \\ \text{and } C \subset P(U) \end{array} \right\}
\end{align*}
holds.
\end{prop}

\begin{proof}
By Remark~\ref{rem4} and Lemma~\ref{lem3}, if $\boldsymbol{V}(C)$ is empty, then the assertion is clear.
Assume that there exists a t-admissible matrix $U$ whose first column $\boldsymbol{u}_1$ is a nonzero vector and satisfying $C \subset P(U)$ and $l \colonequals \operatorname{rank}U$.
Then the first nonzero entry of $\boldsymbol{u}_1$ is positive as $U$ is t-admissible.
If $T = \{ (f, g) \}$ is a finite congruence tropical basis of $C$, then since $\langle T \rangle \subset C \subset P(U)$ and $T$ is finite, by Lemma~\ref{lem3} and Corollary~\ref{cor2}, there exists a t-admissible matrix $U^{\prime}$ whose $(1, 1)$-entry is one and which satisfies $\operatorname{rank}U^{\prime} = l$ and $\langle T \rangle \subset P(U^{\prime})$.
By applying Lemma~\ref{lemm4} to $U^{\prime}$ repeatedly, it is enough to consider the case that the first column of $U^{\prime}$ is $\boldsymbol{e}_1$:
\begin{align*}
    U^{\prime} \equalscolon \begin{pmatrix}
        1 & \vec{w}\\
        0 & \vec{u}_1\\
        \vdots & \vdots\\
        0 & \vec{u}_{l-1}
    \end{pmatrix}.
\end{align*}
For $\varepsilon_i > 0$, by Lemmas~\ref{lem3} and \ref{lemm4}, the matrix
\begin{align*}
    U^{\prime}_{(\varepsilon_1, \ldots, \varepsilon_{l-1})} \colonequals
    \begin{pmatrix*}[l]
        1 & \vec{w}\\
        1 & \vec{w} + \varepsilon_1\vec{u}_1\\
        \,\vdots & \,\vdots\\
        1 & \vec{w} + \varepsilon_1 \vec{u}_1 + \cdots + \varepsilon_{l-1} \vec{u}_{l-1}
    \end{pmatrix*}
\end{align*}
is t-admissible and $P(U^{\prime}) = P(U^{\prime}_{(\varepsilon_1, \ldots, \varepsilon_{l-1})})$ holds.
As $(f, g)$ is in the prime congruence $P(U^{\prime})$, it is also in $P(U^{\prime}(1))$.
Hence there exist terms $m_f$ and $m_g$ of $f$ and $g$, respectively, such that $m_f(\boldsymbol{w}) = f(\boldsymbol{w}) = g(\boldsymbol{w}) = m_g(\boldsymbol{w})$, where $\boldsymbol{w} \colonequals \t \vec{w}$.
These terms $m_f$ and $m_g$ may not be unique, but, as they are continuous, they can be chosen to satisfy
\begin{align*}
&~f \left( \boldsymbol{w} + \sum_{i = 1}^j \delta_i \boldsymbol{u}_i \right) = m_f \left( \boldsymbol{w} + \sum_{i = 1}^j \delta_i \boldsymbol{u}_i \right)\\
=&~ m_g \left( \boldsymbol{w} + \sum_{i = 1}^j \delta_i \boldsymbol{u}_i \right) = g \left( \boldsymbol{w} + \sum_{i = 1}^j \delta_i \boldsymbol{u}_i \right),
\end{align*}
where $\boldsymbol{u}_i \colonequals \t \vec{u}_i$, for any $j = 1, \ldots, l$ with any sufficiently small positive numbers $\delta_1, \ldots, \delta_l$.
Therefore $\boldsymbol{V}(C) = \boldsymbol{V}(T)$ contains the $(l - 1)$-dimensional $\boldsymbol{R}$-rational polyhedral set $\left\{ \boldsymbol{w} + \sum_{i = 1}^{l - 1} \delta_i \boldsymbol{u}_i  \,\middle|\, 0 \le \delta_i \le \delta \right\}$, where $\delta$ is a sufficiently small positive number.
\end{proof}

Propositions~\ref{prop2} and \ref{prop3} give us the following:

\begin{cor}
	\label{cor3}
Let $C$ be a congruence on $\boldsymbol{T}[X_1^{\pm}, \ldots, X_n^{\pm}]$ having a finite congruence tropical basis.
If $\boldsymbol{V}(C)$ is nonempty, then the equalities
\begin{align*}
\operatorname{dim}\boldsymbol{V}(C) + 1 &= \operatorname{max} \left\{ \operatorname{rank}U \,\middle|\, \begin{array}{l} U \text{ is a t-admissible matrix such that}\\ \text{the first column of } U \text{ is } \boldsymbol{e}_1 \\ \text{and } C \subset P(U) \end{array} \right\}\\
&= \operatorname{max} \left\{ \operatorname{rank}U \,\middle|\, \begin{array}{l} U \text{ is a t-admissible matrix such that}\\ \text{the first column of } U \text{ is a nonzero vector} \\ \text{and } C \subset P(U) \end{array} \right\}
\end{align*}
hold.
\end{cor}

By Corollary~\ref{cor3}, the following is clear:

\begin{cor}
	\label{cor4}
If $P$ is a prime congruence on $\boldsymbol{T}[X_1^{\pm}, \ldots, X_n^{\pm}]$ such that $\boldsymbol{V}(P)$ consists of only one point and $\operatorname{dim}\boldsymbol{T}[X_1^{\pm}, \ldots, X_n^{\pm}] / P \ge 2$, then it has no finite congruence tropical bases.
\end{cor}

\begin{prop}
	\label{prop4}
Let $C$ be a proper congruence on $\boldsymbol{T}[X_1^{\pm}, \ldots, X_n^{\pm}]$ and $C_{\boldsymbol{B}}$ the congruence $\langle \{ (f_{\boldsymbol{B}}, g_{\boldsymbol{B}}) \,|\, (f, g) \in C \} \rangle$ on $\boldsymbol{T}[X_1^{\pm}, \ldots, X_n^{\pm}]$.
If $C_{\boldsymbol{B}}$ has a finite congruence tropical basis, then the equality
\begin{align*}
\operatorname{dim}\boldsymbol{V}(C_{\boldsymbol{B}}) = \operatorname{max} \left\{ \operatorname{rank}U \,\middle|\, \begin{array}{l} U \text{ is a t-admissible matrix such that}\\ \text{the first column of } U \text{ is the zero vector } \boldsymbol{0} \\ \text{and } C \subset P(U) \end{array} \right\}
\end{align*}
holds.
\end{prop}

\begin{proof}
As $C$ is proper, by Remark~\ref{rem5}, $\boldsymbol{V}(C_{\boldsymbol{B}})$ is nonempty.
Let  $d \colonequals \operatorname{dim}\boldsymbol{V}(C_{\boldsymbol{B}})$.
Since $\boldsymbol{V}(C_{\boldsymbol{B}})$ has a finite congruence tropical basis, it contains a $d$-dimensional $\boldsymbol{R}$-rational polyhedral set $Y$ satisfying $\boldsymbol{0} \in Y$ and there exist linearly independent $\boldsymbol{y}_1, \ldots, \boldsymbol{y}_d \in \boldsymbol{Z}^n$ and $\varepsilon_1, \ldots, \varepsilon_d > 0$ such that $\sum_{i = 1}^d \delta_i \boldsymbol{y}_i \in Y$ for any $0 \le \delta_i \le \varepsilon_i$ with $i = 1, \ldots, d$.
For these vectors $\boldsymbol{y}_1, \ldots, \boldsymbol{y}_d$, the same argument in the proof of Proposition~\ref{prop2} is true.
Since $\boldsymbol{w}$ (resp.~$\boldsymbol{u}_i$) is replaced with $\boldsymbol{0}$ (resp.~$\boldsymbol{y}_i$), and hence $U$ with
\begin{align*}
        W \colonequals \begin{pmatrix}
        0 & \vec{y}_1\\
        \vdots & \vdots\\
        0 & \vec{y}_d
    \end{pmatrix},
\end{align*}
where $\vec{y}_i \colonequals \t \boldsymbol{y}_i$, and $\boldsymbol{y}_1, \ldots, \boldsymbol{y}_d$ are linearly independent, $\operatorname{rank}W$ is $d$ in this case and $W$ is t-admissible.
As $C_{\boldsymbol{B}} \subset P(W)$ holds, $C \subset P(W)$ holds.
Therefore the inequality
\begin{align*}
\operatorname{dim}\boldsymbol{V}(C_{\boldsymbol{B}}) \le \operatorname{max} \left\{ \operatorname{rank}U \,\middle|\, \begin{array}{l} U \text{ is a t-admissible matrix such that}\\ \text{the first column of } U \text{ is the zero vector } \boldsymbol{0} \\ \text{and } C \subset P(U) \end{array} \right\}
\end{align*}
holds.

Let $U$ be a t-admissible matrix whose first column is the zero vector $\boldsymbol{0}$ and $C \subset P(U)$.
These conditions ensure $C_{\boldsymbol{B}} \subset P(U)$.
In fact, it is easy to check that $(f, g) \in P(U)$ if and only if $(f_{\boldsymbol{B}}, g_{\boldsymbol{B}}) \in P(U)$ for $f, g \in \boldsymbol{T}[X_1^{\pm}, \ldots, X_n^{\pm}]$.
Let $T = \{ (f, g) \}$ be a finite tropical basis of $C_{\boldsymbol{B}}$.
Since $(f, g)$ is in $P(U)$, it is also in $P(U(1))$.
Hence if
\begin{align*}
    U \equalscolon \begin{pmatrix}
        0 & \vec{u}_1\\
        \vdots & \vdots\\
        0 & \vec{u}_l
    \end{pmatrix},
\end{align*}
then $f(\boldsymbol{u}_1) = g(\boldsymbol{u}_1)$ holds, where $l \colonequals \operatorname{rank}U$ and $\boldsymbol{u}_1 \colonequals \t \vec{u}_1$.
By the same argument of the proof of Proposition~\ref{prop4}, an $l$-dimensional $\boldsymbol{R}$-rational polyhedral set contained in $\boldsymbol{V}(C_{\boldsymbol{B}})$ can be found.
Thus the inequality
\begin{align*}
\operatorname{dim}\boldsymbol{V}(C_{\boldsymbol{B}}) \ge \operatorname{max} \left\{ \operatorname{rank}U \,\middle|\, \begin{array}{l} U \text{ is a t-admissible matrix such that}\\ \text{the first column of } U \text{ is the zero vector } \boldsymbol{0} \\ \text{and } C \subset P(U) \end{array} \right\}
\end{align*}
holds.
\end{proof}

Theorem~~\ref{thm:main1-2} follows Remark~\ref{rem2} and Propositions~\ref{prop2}, \ref{prop3}, \ref{prop4}.

By Theorem~\ref{thm:main1-2} and Lemma~\ref{lem2}, the following holds:

\begin{cor}
	\label{cor5}
Let $C$ be a finitely generated proper congruence on $\boldsymbol{T}[X_1^{\pm}, \ldots, X_n^{\pm}]$.
Then the equality
\begin{align*}
    \operatorname{dim} \boldsymbol{T}[ X_1^{\pm}, \ldots, X_n^{\pm}] / C = \operatorname{max} \{ \operatorname{dim} \boldsymbol{V}(C) + 1, \operatorname{dim} \boldsymbol{V}(C_{\boldsymbol{B}}) \}
\end{align*}
holds.
\end{cor}

The class of tropical ideals in the tropical Laurent polynomial semiring $\boldsymbol{T}[X_1^{\pm}, \ldots, X_n^{\pm}]$, which is introduced in \cite[Definition~1.1]{Maclagan=Rincon3}, strictly contains all tropicalizations of ideals in the Laurent polynomial ring $K[X_1^{\pm}, \ldots, X_n^{\pm}]$ over a valuation field $K$.
It was shown in \cite{Maclagan=Rincon2} and \cite{Maclagan=Rincon3} that every tropical ideal $I$ has its dimension $\operatorname{dim}I$ and $\operatorname{Bend}(I)$ has a finite congruence tropical basis and $\boldsymbol{V}(\operatorname{Bend}(I))$ has dimension $\operatorname{dim}I$.
If $I_{\boldsymbol{B}}$ denotes $\{ f_{\boldsymbol{B}} \,|\, f \in I \}$ and we consider this sits in the Laurent polynomial semiring $\boldsymbol{B}[X_1^{\pm} ,\ldots, X_n^{\pm}]$ over $\boldsymbol{B}$, then it is also its tropical ideal, and hence the same things hold for $I_{\boldsymbol{B}}$.
Moreover, by \cite[Proposition~2.12]{Maclagan=Rincon2}, when we consider $I_{\boldsymbol{B}}$ sits in $\boldsymbol{T}[X_1^{\pm} ,\ldots, X_n^{\pm}]$, the congruence variety $\boldsymbol{V}(\operatorname{Bend}(I_{\boldsymbol{B}}))$ is the recession fan of $\boldsymbol{V}(\operatorname{Bend}(I))$.
This fact gives the following corollary:

\begin{cor}
	\label{cor6}
For any proper tropical ideal $I$ in $\boldsymbol{T}[X_1^{\pm}, \ldots, X_n^{\pm}]$, the equality
\begin{align*}
    \operatorname{dim}\boldsymbol{T}[X_1^{\pm}, \ldots, X_n^{\pm}] / \operatorname{Bend}(I) = \operatorname{dim}I + 1
\end{align*}
holds.
\end{cor}

In \cite{Dress=Wenzel}, Dress and Wenzel defined valuated matroids, which are matroids with specific functions.
Speyer introduced in \cite{Speyer} the concept of tropical linear spaces as intersections of a kind of finite numbers of tropical hyperplanes, and verified that valuated matroids with finite ground sets define tropical linear spaces (in the following manner) and vice versa.
For a valuated matroid $p$, each of its circuits defines a tropical linear form.
When the ground set of $p$ is finite, there are only a finite number of such tropical linear forms, and the intersection of all associated tropical hyperplanes (in $\boldsymbol{R}^n$) is a tropical linear space.
When $p$ has a finite ground set $E$, then by taking the support of $p$, a matroid $M$ over $E$ whose rank coincides with that of $p$ is constructed.
Then the tropical linear spaces associated with $M$ are the recession cones of tropical linear spaces associated with $p$, which make the recession fan.
Thus we have the following:

\begin{cor}
	\label{cor7}
Let $p$ be a valuated matroid with a finite ground set, $\mathcal{C}_p$ the set of tropical linear forms defined by circuits of $p$, and $L_p$ the tropical linear space associated with $p$ in $\boldsymbol{R}^n$.
Then the equalities
\begin{align*}
    &~ \operatorname{dim} \boldsymbol{T}[X_1^{\pm}, \ldots, X_n^{\pm}] / \operatorname{Bend}(\mathcal{C}_p)\\
    =&~ \operatorname{dim} \boldsymbol{T}[X_1^{\pm}, \ldots, X_n^{\pm}] / \boldsymbol{E}_{\pm}(L_p)\\
    =&~ \operatorname{dim} L_p + 1
\end{align*}
hold.
\end{cor}

Next, we prove another of our main theorems:

\begin{thm}[Theorem~\ref{thm:main2-1}]
    \label{thm:main2-2}
Let $\Gamma$ be a tropical curve.
Then the equalities
\begin{align*}
    \operatorname{dim}\operatorname{Rat}(\Gamma) = \begin{cases} 1 \quad \text{if $\Gamma$ is one point},\\
    2 \quad \text{otherwise}
    \end{cases}
\end{align*}
hold.
\end{thm}

With Theorem~\ref{thm:main2-2}, we can replace condition $(4)$ of \cite[Corollary~3.19]{JuAe7} with the condition that $\operatorname{dim}S$ is one or two:

\begin{cor}[{\cite[Corollary~3.19]{JuAe7}}]
    \label{cor:curve}
Let $S$ be a finitely generated semifield over $\boldsymbol{T}$.
Then $S$ is isomorphic to $\operatorname{Rat}(\Gamma)$ as a $\boldsymbol{T}$-algebra with some tropical curve $\Gamma$ if and only if some (and any) surjective $\boldsymbol{T}$-algebra homomorphism $\psi$ from a tropical rational function semifield to $S$ satisfies the following five conditions:

$(1)$ $\operatorname{Ker}(\psi)$ is finitely generated as a congruence,

$(2)$ $\operatorname{Ker}(\psi) = \boldsymbol{E}(\boldsymbol{V}(\operatorname{Ker}(\psi)))$ holds,

$(3)$ $\boldsymbol{V}(\operatorname{Ker}(\psi))$ is connected,

$(4)$ $\operatorname{dim}(\boldsymbol{V}(\operatorname{Ker}(\psi)))$ is zero or one, and

$(5)$ $\boldsymbol{V}(\operatorname{Ker}(\psi))$ has no two distinct rays whose primitive direction vectors coincide.
\end{cor}

Let us prepare three lemmas to prove Theorem~\ref{thm:main2-2}.
First we reveal the relation between congruences on $\boldsymbol{T}[X_1^{\pm}, \ldots, X_n^{\pm}]$ and those on $\overline{\boldsymbol{T}(X_1, \ldots, X_n)}$ of the forms $\boldsymbol{E}_{\pm}(V)$ and $\boldsymbol{E}(V)$.
Let $V$ be a subset of $\boldsymbol{R}^n$.

\begin{lemma}
    \label{lem:cancellative}
The $\boldsymbol{B}$-algebra $\boldsymbol{T}[X_1^{\pm}, \ldots, X_n^{\pm}] / \boldsymbol{E}_{\pm}(V)$ is cancellative.
\end{lemma}

\begin{proof}
If $V$ is empty, then $\boldsymbol{E}_{\pm}(\varnothing) = \boldsymbol{T}[X_1^{\pm}, \ldots, X_n^{\pm}]^2$, and thus the $\boldsymbol{B}$-algebra $\boldsymbol{T}[X_1^{\pm}, \ldots, X_n^{\pm}] / \boldsymbol{E}_{\pm}(\varnothing) = \{ -\infty \}$ is cancellative.
Assume that $V$ is nonempty.
For $f, g, h \in \boldsymbol{T}[X_1^{\pm}, \ldots, X_n^{\pm}] / \boldsymbol{E}_{\pm}(V)$ such that $f \not= -\infty$, if $f \odot g = f \odot h$ holds, then $g(x) = h(x)$ holds for any $x \in V$ since $f(x) \in \boldsymbol{R}$, which completes the proof.
\end{proof}

It is clear that $Q(\boldsymbol{T}[X_1^{\pm}, \ldots, X_n^{\pm}] / \boldsymbol{E}_{\pm}(\boldsymbol{R}^n)) = \overline{\boldsymbol{T}(X_1, \ldots, X_n)}$ holds.

If $V$ is empty, then $\boldsymbol{E}_{\pm}(\varnothing) = \boldsymbol{T}[X_1^{\pm}, \ldots, X_n^{\pm}]^2$ and $\boldsymbol{E}(\varnothing) = \overline{\boldsymbol{T}(X_1, \ldots, X_n)}^2$ hold.
Thus in this case, $\boldsymbol{T}[X_1^{\pm}, \ldots, X_n^{\pm}] / \boldsymbol{E}_{\pm}(\varnothing) = \{ -\infty \}$ and $\overline{\boldsymbol{T}(X_1, \ldots, X_n)} / \boldsymbol{E}(\varnothing) = \{ -\infty \}$, and hence $Q(\boldsymbol{T}[X_1^{\pm}, \ldots, X_n^{\pm}] / \boldsymbol{E}_{\pm}(\varnothing)) = \{ -\infty \} = \overline{\boldsymbol{T}(X_1, \ldots, X_n)} / \boldsymbol{E}(\varnothing)$ hold.
Note that clearly $\langle \boldsymbol{E}_{\pm}(\varnothing) / \boldsymbol{E}_{\pm}(\boldsymbol{R}^n) \rangle_{\overline{\boldsymbol{T}(X_1, \ldots, X_n)}} = \boldsymbol{E}(\varnothing)$ holds.
If $V$ is nonempty, then $(f, -\infty) \in \boldsymbol{E}_{\pm}(V)$ implies $f = -\infty$.
Therefore, by \cite[Proposition~3.9(iv)]{Joo=Mincheva1}, we have 
\begin{align*}
&~Q(\boldsymbol{T}[X_1^{\pm}, \ldots, X_n^{\pm}] / \boldsymbol{E}_{\pm}(V))\\
\cong&~ Q((\boldsymbol{T}[X_1^{\pm}, \ldots, X_n^{\pm}] / \boldsymbol{E}_{\pm}(\boldsymbol{R}^n)) / (\boldsymbol{E}_{\pm}(V) / \boldsymbol{E}_{\pm}(\boldsymbol{R}^n)))\\
\cong&~Q(\boldsymbol{T}[X_1^{\pm}, \ldots, X_n^{\pm}] / \boldsymbol{E}_{\pm}(\boldsymbol{R}^n)) / \langle \boldsymbol{E}_{\pm}(V) / \boldsymbol{E}_{\pm}(\boldsymbol{R}^n) \rangle_{\overline{\boldsymbol{T}(X_1, \ldots, X_n)}}\\
=&~\overline{\boldsymbol{T}(X_1, \ldots, X_n)} / \langle \boldsymbol{E}_{\pm}(V) / \boldsymbol{E}_{\pm}(\boldsymbol{R}^n) \rangle_{\overline{\boldsymbol{T}(X_1, \ldots, X_n)}}.
\end{align*}

\begin{lemma}
	\label{lem5}
The equality
\begin{align*}
\langle \boldsymbol{E}_{\pm}(V) / \boldsymbol{E}_{\pm}(\boldsymbol{R}^n) \rangle_{\overline{\boldsymbol{T}(X_1, \ldots, X_n)}} = \boldsymbol{E}(V)
\end{align*}
holds.
\end{lemma}

\begin{proof}
The case when $V$ is empty was given above.

Assume that $V$ is nonempty.

For any $(f, g) \in \langle \boldsymbol{E}_{\pm}(V) / \boldsymbol{E}_{\pm}(\boldsymbol{R}^n) \rangle_{\overline{\boldsymbol{T}(X_1, \ldots, X_n)}}$, there exists a finite subset $F$ of $\boldsymbol{E}_{\pm}(V) / \boldsymbol{E}_{\pm}(\boldsymbol{R}^n)$ by Lemma \ref{lem:finite} such that $(f, g) \in \langle F \rangle_{\overline{\boldsymbol{T}(X_1, \ldots, X_n)}}$.
By \cite[Corollary~2.12]{Bertram=Easton}, the value $f(x)$ must be $g(x)$ for any $x \in V$, and hence $(f, g) \in \boldsymbol{E}(V)$.

For any $(f, g) \in \boldsymbol{E}(V)$, there exist $f_1, f_2, g_1, g_2 \in \overline{\boldsymbol{T}[X_1, \ldots, X_n]}$ such that $f_2 \not= -\infty, g_2 \not= -\infty$ and $f = f_1 \odot f_2^{\odot (-1)}, g = g_1 \odot g_2^{\odot (-1)}$.
Since $V$ is nonempty, $f_1 = -\infty$ if and only if $g_1 = -\infty$.
Assume that $f_1 \not= -\infty$ and $g_1 \not= -\infty$.
Since $(f_1 \odot g_2, g_1 \odot f_2) \in \overline{\boldsymbol{T}[X_1, \ldots, X_n]}^2$ and $f_1(x) \odot g_2(x) = g_1(x) \odot f_2(x)$ for any $x \in V$, the set $\boldsymbol{E}_{\pm}(V) / \boldsymbol{E}_{\pm}(\boldsymbol{R}^n)$ contains pairs of preimages of $f_1 \odot g_2$ and $g_1 \odot f_2$ in $\boldsymbol{T}[X_1, \ldots, X_n] \subset \boldsymbol{T}[X_1^{\pm}, \ldots, X_n^{\pm}]$.
In conclusion,
\begin{align*}
(f, g) &= \left((f_1 \odot g_2) \odot \left(f_2^{\odot (-1)} \odot g_2^{\odot (-1)} \right), (g_1 \odot f_2) \odot \left(f_2^{\odot (-1)} \odot g_2^{\odot (-1)} \right)\right)\\
&\in \langle \boldsymbol{E}_{\pm}(V) / \boldsymbol{E}_{\pm}(\boldsymbol{R}^n) \rangle_{\overline{\boldsymbol{T}(X_1, \ldots, X_n)}}    
\end{align*}
hold.
\end{proof}

Therefore, for any subset $V$ of $\boldsymbol{R}^n$, we have
\begin{align*}
Q(\boldsymbol{T}[X_1^{\pm}, \ldots, X_n^{\pm}] / \boldsymbol{E}_{\pm}(V)) \cong \overline{\boldsymbol{T}(X_1, \ldots, X_n)} / \boldsymbol{E}(V).
\end{align*}

Next we consider the relation between $\boldsymbol{V}((f_{\boldsymbol{B}}, g_{\boldsymbol{B}}))$ and the recession fan $\operatorname{rec}(\boldsymbol{V}((f, g)))$ for $f, g \in \boldsymbol{T}[X_1^{\pm}, \ldots, X_n^{\pm}]$.

\begin{lemma}
	\label{lem6}
Let $X$ be a finite union of nonempty $\boldsymbol{R}$-rational polyhedral sets in $\boldsymbol{R}^n$.
Then there exists an element $(f, g)$ in $\boldsymbol{E}(X)$ satisfying $\boldsymbol{V}((f_{\boldsymbol{B}}, g_{\boldsymbol{B}})) = \operatorname{rec}(X)$.
Moreover, $\boldsymbol{V}(\boldsymbol{E}(X)_{\boldsymbol{B}}) = \operatorname{rec}(X)$ holds with the congruence $\boldsymbol{E}(X)_{\boldsymbol{B}} \colonequals \langle \{ (f^{\prime}_{\boldsymbol{B}}, g^{\prime}_{\boldsymbol{B}}) \,|\, (f^{\prime}, g^{\prime}) \in \boldsymbol{E}(X) \} \rangle_{\overline{\boldsymbol{T}(X_1, \ldots, X_n)}}$.
The same things are true for $\boldsymbol{E}_{\pm}(X)$.
\end{lemma}

\begin{proof}
Let $Y_1, \ldots, Y_k$ be nonempty $\boldsymbol{R}$-rational polyhedral sets in $\boldsymbol{R}^n$ such that $X = \bigcup_{i = 1}^k Y_i$.
For each $Y_i$, there exist a matrix $A_i$ with integer entries and a vector $\boldsymbol{b}_i$ such that $Y_i = \{ \boldsymbol{x} \in \boldsymbol{R}^n \,|\, A_i \boldsymbol{x} \ge \boldsymbol{b}_i \}$.
Let
\begin{align*}
A_i \equalscolon \begin{pmatrix} \vec{a}_{i1} \\ \vdots \\ \vec{a}_{il_i} \end{pmatrix} \quad \text{and} \quad \boldsymbol{b}_i \equalscolon \begin{pmatrix} b_{i1} \\ \vdots \\ b_{il_i} \end{pmatrix}.
\end{align*}
If $\boldsymbol{a}_{ij} \colonequals \t \vec{a}_{ij}$ and $f_i \colonequals b_{i1} \odot \boldsymbol{X}^{\odot (- \boldsymbol{a}_{i1})} \odot \cdots \odot b_{il_i} \odot \boldsymbol{X}^{\odot (-\boldsymbol{a}_{il_i})} \oplus 0 \in \overline{\boldsymbol{T}(X_1, \ldots, X_n)}$, then $\boldsymbol{V}((f_i, 0)) = Y_i$ and
\begin{align*}
\boldsymbol{V}((f_i)_{\boldsymbol{B}}, 0_{\boldsymbol{B}}) &= \boldsymbol{V}((\boldsymbol{X}^{\odot (-\boldsymbol{a}_{i1})} \odot \cdots \odot \boldsymbol{X}^{\odot (- \boldsymbol{a}_{ik_i})} \oplus 0, 0))\\
&= \{ \boldsymbol{x} \in \boldsymbol{R}^n \,|\, A_i \boldsymbol{x} \ge \boldsymbol{0} \}\\
&= \operatorname{rec}(Y_i)
\end{align*}
hold.
For
\begin{align*}
F \colonequals \left( \bigoplus_{i = 1}^k f_{i}^{\odot (-1)} \right)^{\odot (-1)},
\end{align*}
it is easy to check that
\begin{align*}
\boldsymbol{V}((F, 0)) &= \bigcup_{i = 1}^k Y_i = X,\\
\boldsymbol{V}((F_{\boldsymbol{B}}, 0)) &= \bigcup_{i = 1}^k \operatorname{rec}(Y_i) = \operatorname{rec}(X)
\end{align*}
hold (cf~\cite[Lemma~3.10(1)]{JuAe7}).
Hence if $(f, g) \colonequals (F, 0)$, then since $(f, g) \in \boldsymbol{E}(X)$, and hence $(f_{\boldsymbol{B}}, g_{\boldsymbol{B}}) \in \boldsymbol{E}(X)_{\boldsymbol{B}}$, the inclusions $\boldsymbol{V}(\boldsymbol{E}(X)_{\boldsymbol{B}}) \subset \boldsymbol{V}((f_{\boldsymbol{B}}, g_{\boldsymbol{B}})) = \operatorname{rec}(X)$ hold.
By \cite[the proof of Theorem~1.1]{JuAe7}, this pair $(f, g)$ generates $\boldsymbol{E}(X)$.
Thus $\boldsymbol{E}(X)_{\boldsymbol{B}}$ coincides with $\langle (f_{\boldsymbol{B}}, g_{\boldsymbol{B}}) \rangle$ by Lemma~\ref{lem2}.
Therefore $\boldsymbol{V}(\boldsymbol{E}(X)_{\boldsymbol{B}}) = \boldsymbol{V}((f_{\boldsymbol{B}}, g_{\boldsymbol{B}})) = \operatorname{rec}(X)$ follow.
In the case of $\boldsymbol{T}[X_1^{\pm}, \ldots, X_n^{\pm}]$, when $F_l \odot \overline{g} = \overline{f}$ with $\overline{f}, \overline{g} \in \overline{\boldsymbol{T}[X_1, \ldots, X_n]} \setminus \{ -\infty \}$, some pair $(f, g)$ of these preimages $f, g \in \boldsymbol{T}[X_1, \ldots, X_n] \setminus \{ -\infty \} \subset \boldsymbol{T}[X_1^{\pm}, \ldots, X_n^{\pm}] \setminus \{ -\infty \}$ is desired.
In fact, by the same argument above, $\boldsymbol{V}((f_{\boldsymbol{B}}, g_{\boldsymbol{B}})) = \operatorname{rec}(X)$, and hence $\boldsymbol{V}(\boldsymbol{E}_{\pm}(X)_{\boldsymbol{B}}) \subset \operatorname{rec}(X)$ hold.
If $\operatorname{rec}(X) \setminus \boldsymbol{V}(\boldsymbol{E}_{\pm}(X)_{\boldsymbol{B}})$ is nonempty, then there exists $(f^{\prime}, g^{\prime}) \in \boldsymbol{E}_{\pm}(X)_{\boldsymbol{B}}$ such that $\operatorname{rec}(X) \setminus \boldsymbol{V}((f^{\prime}, g^{\prime}))$ is nonempty.
If $\overline{f^{\prime}}$ and $\overline{g^{\prime}}$ denote the images of $f^{\prime}$ and $g^{\prime}$ in $\overline{\boldsymbol{T}(X_1, \ldots, X_n)}$, respectively, then $\boldsymbol{V}((\overline{f^{\prime}}, \overline{g^{\prime}})) = \boldsymbol{V}((f^{\prime}, g^{\prime}))$ and $(\overline{f^{\prime}}, \overline{g^{\prime}}) \in \boldsymbol{E}(X)_{\boldsymbol{B}}$, this is a contradiction.
\end{proof}

Now we can prove Theorem~\ref{thm:main2-2}:

\begin{proof}[Proof of Theorem~\ref{thm:main2-2}]
For some $n$, there exists a surjective $\boldsymbol{T}$-algebra homomorphism $\psi \colon \overline{\boldsymbol{T}(X_1, \ldots, X_n)} \twoheadrightarrow \operatorname{Rat}(\Gamma)$ satisfying the five conditions in Corollary~\ref{cor:curve}.
Note that \cite[Proposition~3.12]{JuAe6} ensures that the natural compactification obtained from $\boldsymbol{V}(\operatorname{Ker}(\psi))$ by one-point compactifications of rays is isomorphic (i.e., isometric) to $\Gamma$ with lattice length (see \cite[Subsection~2.7 and after the proof of Lemma~3.17]{JuAe7} for more details).
By \cite[Proposition~3.4]{JuAe7}, there exist $f, g \in \overline{\boldsymbol{T}[X_1, \ldots, X_n]} \setminus \{ -\infty \}$ such that $\boldsymbol{E}(\boldsymbol{V}(\operatorname{Ker}(\psi))) = \left\langle \left( f \odot g^{\odot (-1)}, 0 \right) \right\rangle_{\overline{\boldsymbol{T}(X_1, \ldots, X_n)}}$.
It is easy to check that $\boldsymbol{V} \left( \left( f \odot g^{\odot (-1)}, 0 \right) \right) = \boldsymbol{V}((f, g))$ holds.
Thus, by the previous argument, $\boldsymbol{T}$-algebra isomorphisms
\begin{align*}
    \operatorname{Rat}(\Gamma) &\cong \overline{\boldsymbol{T}(X_1, \ldots, X_n)} / \left\langle \left( f \odot g^{\odot (-1)}, 0 \right) \right\rangle_{\overline{\boldsymbol{T}(X_1, \ldots, X_n)}}\\
    &\cong Q\left( \boldsymbol{T}[X_1^{\pm}, \ldots, X_n^{\pm}] / \boldsymbol{E}_{\pm}(\boldsymbol{V}((f, g))) \right)
\end{align*}
follow.
Since $(h, -\infty) \in C$ implies $h = -\infty$ for any proper congruence $C$ on $\boldsymbol{T}[X_1^{\pm}, \ldots, X_n^{\pm}]$ by Remark~\ref{rem1},
\begin{align*}
    &~ \operatorname{dim} \operatorname{Rat}(\Gamma)\\
    =&~\operatorname{dim} Q\left( \boldsymbol{T}[X_1^{\pm}, \ldots, X_n^{\pm}] / \boldsymbol{E}_{\pm}(\boldsymbol{V}((f, g))) \right)\\
    =&~ \operatorname{dim} \boldsymbol{T}[X_1^{\pm}, \ldots, X_n^{\pm}] / \boldsymbol{E}_{\pm}(\boldsymbol{V}((f, g)))\\
    =&~ \operatorname{dim}\boldsymbol{V}((f, g)) + 1\\
    =&~ \begin{cases} 1 \quad \text{if $\Gamma$ is one point,} \\ 2 \quad \text{otherwise} \end{cases}
\end{align*}
hold  by Lemmas~\ref{lem:cancellative}, \ref{lem6}, \cite[Proposition~3.9(ii),(iii)]{Joo=Mincheva1} and Theorem~\ref{thm:main1-2}.
\end{proof}

The argument above, \cite[Theorem~1.1]{JuAe7}, Lemma~\ref{lem6} and Theorem~\ref{thm:main1-2} also give the following: 

\begin{cor}
	\label{cor8}
Let $V$ be a finite union of $\boldsymbol{R}$-rational polyhedral sets.
If $V$ is nonempty, then the equality
\begin{align*}
    \operatorname{dim} \overline{\boldsymbol{T}(X_1, \ldots, X_n)} / \boldsymbol{E}(V) = \operatorname{dim}V + 1
\end{align*}
holds.
\end{cor}

Conditions $(3)$ and $(5)$ in Corollary~\ref{cor:curve} are only depend on the $\boldsymbol{T}$-algebra structure of $S$, but are written in the terms of geometric object $\boldsymbol{V}(\operatorname{Ker}(\psi))$.
The remains of this paper are devoted to replacing these two conditions with algebraic conditions.

We consider the relation between the connectivity of $V \subset \boldsymbol{R}^n$ and the $\boldsymbol{T}$-algebra structure of $\overline{\boldsymbol{T}(X_1, \ldots, X_n)}/ \boldsymbol{E}(V)$. 

\begin{dfn}[cf.~{\cite[Chapter~2]{Golan}}]
	\label{dfn:pseudodirect product}
\upshape{
For two semifields $S_1$ and $S_2$ over $\boldsymbol{T}$, we define
\begin{align*}
S_1 \bowtie S_2 \colonequals \{ (s_1, s_2) \in (S_1 \setminus \{ 0_{S_1} \}) \times (S_2 \setminus \{ 0_{S_2} \}) \} \cup \{ (0_{S_1}, 0_{S_2}) \}.
\end{align*}
}
We write it also $\bowtie_{i = 1}^2 S_i$.
For any number (possibly infinite) of semifields over $\boldsymbol{T}$, we define the same things.
\end{dfn}

The following three lemmas clearly hold.

\begin{lemma}
    \label{lem7}
With two natural operations
\begin{align*}
(s_1, s_2) + (t_1, t_2) &\colonequals (s_1 + t_1, s_2 + t_2),\\
(s_1, s_2) \cdot (t_1, t_2) &\colonequals (s_1 \cdot t_1, s_2 \cdot t_2),
\end{align*}
the set $S_1 \bowtie S_2$ becomes a semifield over $\boldsymbol{T}$ with the diagonal semiring homomorphism $\boldsymbol{T} \hookrightarrow S_1 \bowtie S_2; t \mapsto (t, t)$.
\end{lemma}

\begin{lemma}
    \label{lem8}
The projection $\pi_i \colon S_1 \bowtie S_2 \twoheadrightarrow S_i ; (s_1, s_2) \mapsto s_i$ is a $\boldsymbol{T}$-algebra homomorphism by the diagonal semiring homomorphism $\boldsymbol{T} \hookrightarrow S_1 \bowtie S_2; t \mapsto (t, t)$.
\end{lemma}

\begin{lemma}
    \label{lem9}
Both $S_1$ and $S_2$ are finitely generated as semifields over $\boldsymbol{T}$ if and only if so is $S_1 \bowtie S_2$ via the diagonal semiring homomorphism $\boldsymbol{T} \hookrightarrow S_1 \bowtie S_2; t \mapsto (t, t)$.
\end{lemma}

Let $S_1$ and $S_2$ be semifields over $\boldsymbol{T}$ and $S \colonequals S_1 \bowtie S_2$, which is a semifield over $\boldsymbol{T}$ with the diagonal semiring homomorphism $d \colon \boldsymbol{T} \hookrightarrow S; t \mapsto (t, t)$ by Lemma~\ref{lem7}.
Assume that $S$ is finitely generated as a semifield over $\boldsymbol{T}$ via $d$.
Let $\psi$ be a surjective  $\boldsymbol{T}$-algebra homomorphism from a tropical rational function semifield $\overline{\boldsymbol{T}(X_1, \ldots, X_n)} = \overline{\boldsymbol{T}(\boldsymbol{X})}$ to $S$ and $\pi_i \colon S \twoheadrightarrow S_i$ the projection, which is a $\boldsymbol{T}$-algebra homomorphism by Lemma~\ref{lem8}.
Then $\psi$ and the composition $\psi_i \colonequals \pi_i \circ \psi$ induce the natural $\boldsymbol{T}$-algebra isomorphisms $\overline{\psi} \colon S \to \overline{\boldsymbol{T}(\boldsymbol{X})} / \operatorname{Ker}(\psi)$ and $\overline{\psi_i} \colon S_i \to \overline{\boldsymbol{T}(\boldsymbol{X})} / \operatorname{Ker}(\psi_i)$ such that $\overline{\psi} \circ \psi = \pi_{\operatorname{Ker}(\psi)}$ and $\overline{\psi_i} \circ \psi_i = \pi_{\operatorname{Ker}(\psi_i)}$ hold, where $\pi_{\operatorname{Ker}(\psi)} \colon \overline{\boldsymbol{T}(\boldsymbol{X})} \twoheadrightarrow \overline{\boldsymbol{T}(\boldsymbol{X})} / \operatorname{Ker}(\psi)$ and $\pi_{\operatorname{Ker}(\psi_i)} \colon \overline{\boldsymbol{T}(\boldsymbol{X})} \twoheadrightarrow \overline{\boldsymbol{T}(\boldsymbol{X})} / \operatorname{Ker}(\psi_i)$ are the natural surjective $\boldsymbol{T}$-algebra homomorphisms.
These $\boldsymbol{T}$-algebra homomorphisms also induce a surjective $\boldsymbol{T}$-algebra homomorphism $\overline{\pi_i} \colon \overline{\boldsymbol{T}(\boldsymbol{X})} / \operatorname{Ker}(\psi) \twoheadrightarrow \overline{\boldsymbol{T}(\boldsymbol{X})} / \operatorname{Ker}(\psi_i)$ such that $\pi_i \circ \overline{\psi_i} = \overline{\pi_i} \circ \overline{\psi}$ holds for each $i = 1, 2$.

\begin{lemma}
    \label{lem10}
In the above setting, $\operatorname{Ker}(\psi) = \bigcap_{i = 1}^2 \operatorname{Ker}(\psi_i)$ holds.
In particular, $\boldsymbol{V}(\operatorname{Ker}(\psi)) \supset \bigcup_{i = 1}^2 \boldsymbol{V}(\operatorname{Ker}(\psi_i))$ holds.
\end{lemma}

\begin{proof}
For any $(f, g) \in \bigcap_{i = 1}^2 \operatorname{Ker}(\psi_i)$, by definition, $\psi_1(f) = \psi_1(g)$ and $\psi_2(f) = \psi_2(g)$, and hence $\psi(f) = \psi(g)$.
This means that $\operatorname{Ker}(\psi) \supset \bigcap_{i = 1}^2 \operatorname{Ker}(\psi_i)$ holds.
The converse inclusion is clear.
\end{proof}

\begin{lemma}
    \label{lem11}
In the above setting, $\operatorname{Ker}(\psi_1) \rtimes \operatorname{Ker}(\psi_2) \subset \operatorname{Ker}(\psi)$ holds.
In particular, $\boldsymbol{V}(\operatorname{Ker}(\psi)) \subset \bigcup_{i = 1}^2 \boldsymbol{V}(\operatorname{Ker}(\psi_i))$ holds.
\end{lemma}

\begin{proof}
By \cite[Proposition~2.2(1)]{Joo=Mincheva2}, the inclusion $\operatorname{Ker}(\psi_1) \rtimes \operatorname{Ker}(\psi_2) \subset \operatorname{Ker}(\psi_i)$ holds, and hence $\operatorname{Ker}(\psi_1) \rtimes \operatorname{Ker}(\psi_2) \subset \bigcap_{i = 1}^2 \operatorname{Ker}(\psi_i) = \operatorname{Ker}(\psi)$ follows by Lemma~\ref{lem10}.
The last assertion is clear by \cite[Proposition~3.8(3)]{JuAe6}.
\end{proof}

Lemmas~\ref{lem10} and \ref{lem11} induce the following corollary:

\begin{cor}
    \label{cor:union}
In the above setting, $\boldsymbol{V}(\operatorname{Ker}(\psi)) = \bigcup_{i = 1}^2 \boldsymbol{V}(\operatorname{Ker}(\psi_i))$ holds.
\end{cor}

\begin{lemma}
    \label{lem12}
In the above setting, $\bigcap_{i = 1}^2 \boldsymbol{V}(\operatorname{Ker}(\psi_i)) = \varnothing$ holds.
\end{lemma}

\begin{proof}
Since $\psi$ is surjective, by the definition of $S$, there exists $f \in \overline{\boldsymbol{T}(\boldsymbol{X})}$ such that $\psi(f) = (1, 0)$.
By definition, $\pi_1(\psi(f)) = 1$ and $\pi_2(\psi(f)) = 0$ hold.
Hence $\pi_{\operatorname{Ker}(\psi_1)}(f) = \overline{\psi_1}(\pi_1(\psi(f))) = 1$ and $\pi_{\operatorname{Ker}(\psi_2)}(f) = \overline{\psi_2}(\pi_2(\psi(f))) = 0$ hold.
If $\bigcap_{i = 1}^2 \boldsymbol{V}(\operatorname{Ker}(\psi_i))$ contains a point $x$, then $1 = \pi_{\operatorname{Ker}(\psi_1)}(f)(x) = f(x) = \pi_{\operatorname{Ker}(\psi_2)}(f)(x) = 0$, which is a contradiction.
\end{proof}

Since $\boldsymbol{V}(\operatorname{Ker}(\psi_i))$ is closed, we have the following corollary by Corollary~\ref{cor:union} and Lemma~\ref{lem12}:

\begin{cor}
    \label{cor:connected}
In the above setting, $\boldsymbol{V}(\operatorname{Ker}(\psi))$ is the disjoint union $\bigsqcup_{i = 1}^2 \boldsymbol{V}(\operatorname{Ker}(\psi_i))$.
In particular, if both $\boldsymbol{V}(\operatorname{Ker}(\psi_1))$ and $\boldsymbol{V}(\operatorname{Ker}(\psi_2))$ are nonempty, then $\boldsymbol{V}(\operatorname{Ker}(\psi))$ is disconnected.
\end{cor}

\begin{lemma}
    \label{lem13}
In the above setting, if $\operatorname{Ker}(\psi_i) = \boldsymbol{E}(\boldsymbol{V}(\operatorname{Ker}(\psi_i)))$ holds for $i = 1, 2$, then $\operatorname{Ker}(\psi) = \boldsymbol{E}(\boldsymbol{V}(\operatorname{Ker}(\psi)))$ holds.
\end{lemma}

\begin{proof}
For any $(f, g) \in \boldsymbol{E}(\boldsymbol{V}(\operatorname{Ker}(\psi)))$ and $x \in \boldsymbol{V}(\operatorname{Ker}(\psi))$, the value $f(x)$ coincides with $g(x)$.
Thus $(f, g) \in \boldsymbol{E}(\boldsymbol{V}(\operatorname{Ker}(\psi_i))) = \operatorname{Ker}(\psi_i)$ by Corollary~\ref{cor:union}.
This means that $\psi_i(f) = \psi_i(g)$, and thus $\psi(f) = \psi(g)$.
In conclusion, $(f, g) \in \operatorname{Ker}(\psi)$ holds.
\end{proof}

\begin{lemma}
    \label{lem14}
In the above setting, $\overline{\boldsymbol{T}(\boldsymbol{X})} / \boldsymbol{E}(\boldsymbol{V}(\operatorname{Ker}(\psi)))$ is isomorphic to $\bowtie_{i = 1}^2 \overline{\boldsymbol{T}(\boldsymbol{X})} / \boldsymbol{E}(\boldsymbol{V}(\operatorname{Ker}(\psi_i)))$ as a $\boldsymbol{T}$-algebra with the diagonal semiring homomorphism $\boldsymbol{T} \hookrightarrow \bowtie_{i = 1}^2 \overline{\boldsymbol{T}(\boldsymbol{X})} / \boldsymbol{E}(\boldsymbol{V}(\operatorname{Ker}(\psi_i))); t \mapsto (t, t)$.
\end{lemma}

\begin{proof}
The facts $\operatorname{Ker}(\psi) \subset \boldsymbol{E}(\boldsymbol{V}(\operatorname{Ker}(\psi)))$ and $\operatorname{Ker}(\psi_i) \subset \boldsymbol{E}(\boldsymbol{V}(\operatorname{Ker}(\psi_i)))$ induce natural surjective $\boldsymbol{T}$-algebra homomorphsims $\theta \colon \overline{\boldsymbol{T}(\boldsymbol{X})} / \operatorname{Ker}(\psi) \twoheadrightarrow \overline{\boldsymbol{T}(\boldsymbol{X})} / \boldsymbol{E}(\boldsymbol{V}(\operatorname{Ker}(\psi)))$ and $\theta_i \colon \overline{\boldsymbol{T}(\boldsymbol{X})} / \operatorname{Ker}(\psi_i) \twoheadrightarrow \overline{\boldsymbol{T}(\boldsymbol{X})} / \boldsymbol{E}(\boldsymbol{V}(\operatorname{Ker}(\psi_i)))$.
There exists a surjective $\boldsymbol{T}$-algebra homomorphism
\begin{align*}
\phi_i \colon \overline{\boldsymbol{T}(\boldsymbol{X})} / \boldsymbol{E}(\boldsymbol{V}(\operatorname{Ker}(\psi))) \twoheadrightarrow \overline{\boldsymbol{T}(\boldsymbol{X})} / \boldsymbol{E}(\boldsymbol{V}(\operatorname{Ker}(\psi_i)))
\end{align*}
satisfying $\theta_i \circ \overline{\pi_i} = \phi_i \circ \theta$.
In fact, $\theta(\pi_{\operatorname{Ker}(\psi)}(f)) = \theta(\pi_{\operatorname{Ker}(\psi)}(g))$ implies $(f, g) \in \boldsymbol{E}(\boldsymbol{V}(\operatorname{Ker}(\psi)))$.
Hence $f(x) = g(x)$ holds for any $x \in \boldsymbol{V}(\operatorname{Ker}(\psi)) = \bigcup_{i = 1}^2 \boldsymbol{V}(\operatorname{Ker}(\psi_i))$ by Corollary~\ref{cor:union}.
Therefore $(f, g) \in \boldsymbol{E}(\boldsymbol{V}(\operatorname{Ker}(\psi_i)))$, and thus $\theta_i(\overline{\pi_i}(\pi_{\operatorname{Ker}(\psi)}(f))) = \theta_i(\overline{\pi_i}(\pi_{\operatorname{Ker}(\psi)}(g)))$ holds.

Let
\begin{align*}
\phi \colon \overline{\boldsymbol{T}(\boldsymbol{X})} / \boldsymbol{E}(\boldsymbol{V}(\operatorname{Ker}(\psi))) \to \bowtie_{i = 1}^2 \overline{\boldsymbol{T}(\boldsymbol{X})} / \boldsymbol{E}(\boldsymbol{V}(\operatorname{Ker}(\psi_i)))
\end{align*}
be the map defined by the correspondence $[f] \mapsto (\phi_1([f]), \phi_2([f]))$ for any $f \in \overline{\boldsymbol{T}(\boldsymbol{X})}$, where $[f]$ stands for $\theta(\pi_{\operatorname{Ker}(\psi)}(f))$.
As both $\phi_1$ and $\phi_2$ are $\boldsymbol{T}$-algebra homomorphisms, so is $\phi$.
For any $(f, g) \in \overline{\boldsymbol{T}(\boldsymbol{X})}^2$ such that $\phi([f]) = \phi([g])$, by definition, $\phi_1([f]) = \phi_1([g])$ and $\phi_2([f]) = \phi_2([g])$ hold.
Thus for any $x \in \bigcup_{i = 1}^2 \boldsymbol{V}(\operatorname{Ker}(\psi_i)) = \boldsymbol{V}(\operatorname{Ker}(\psi))$, the equality $f(x) = g(x)$ holds.
This means that $(f, g) \in \boldsymbol{E}(\boldsymbol{V}(\operatorname{Ker}(\psi)))$.
Therefore $\phi$ is injective.
Since $S = S_1 \bowtie S_2$ and $\overline{\psi}$ and each $\overline{\psi_i}$ are isomorphisms, $\overline{\boldsymbol{T}(\boldsymbol{X})} / \operatorname{Ker}(\psi)$ is isomorphic to $\bowtie_{i = 1}^2 \overline{\boldsymbol{T}(\boldsymbol{X})} / \operatorname{Ker}(\psi_i)$ as a $\boldsymbol{T}$-algebra with the diagonal semiring homomorphism $\boldsymbol{T} \hookrightarrow \bowtie_{i = 1}^2 \overline{\boldsymbol{T}(\boldsymbol{X})} / \operatorname{Ker}(\psi_i); t \mapsto (t, t)$, and in fact, the map $\overline{\pi} \colon \overline{\boldsymbol{T}(\boldsymbol{X})} / \operatorname{Ker}(\psi) \to \bowtie_{i = 1}^2 \overline{\boldsymbol{T}(\boldsymbol{X})} / \operatorname{Ker}(\psi_i); f \mapsto (\overline{\pi_1}(f), \overline{\pi_2}(f))$ is a $\boldsymbol{T}$-algebra isomorphism.
Hence the surjectivities of $\theta$ and $\theta_i$ ensure that of $\phi$.
\end{proof}

\begin{lemma}
    \label{lem15}
In the above setting, $\operatorname{Ker}(\psi) = \boldsymbol{E}(\boldsymbol{V}(\operatorname{Ker}(\psi)))$ implies $\operatorname{Ker}(\psi_i) = \boldsymbol{E}(\boldsymbol{V}(\operatorname{Ker}(\psi_i)))$ for $i = 1, 2$.
\end{lemma}

\begin{proof}
With the same notations in the proof of Lemma \ref{lem14}, since $\theta$, $\overline{\pi}$ and $\phi$ are $\boldsymbol{T}$-algebra isomorphisms, so is the composition $\phi \circ \theta \circ \overline{\pi}^{-1}$.
It is easy to check that $\phi \circ \theta \circ \overline{\pi}^{-1}$ coincides with the map $\bowtie_{i = 1}^2 \overline{\boldsymbol{T}(\boldsymbol{X})} / \operatorname{Ker}(\psi_i) \to \bowtie_{i = 1}^2 \overline{\boldsymbol{T}(\boldsymbol{X})} / \boldsymbol{E}(\boldsymbol{V}(\operatorname{Ker}(\psi_i))); (f, g) \mapsto (\theta_1(f), \theta_2(g))$.
This means that both $\theta_1$ and $\theta_2$ are $\boldsymbol{T}$-algebra isomorphisms, which completes the proof.
\end{proof}

\begin{cor}
	\label{cor9}
In the above setting, $\operatorname{Ker}(\psi) = \boldsymbol{E}(\boldsymbol{V}(\operatorname{Ker}(\psi)))$ if and only if $\operatorname{Ker}(\psi_i) = \boldsymbol{E}(\boldsymbol{V}(\operatorname{Ker}(\psi_i)))$ for $i = 1, 2$.
Then $\boldsymbol{V}(\operatorname{Ker}(\psi_i))$ is nonempty and $\boldsymbol{V}(\operatorname{Ker}(\psi)) = \bigsqcup_{i = 1}^2 \boldsymbol{V}(\operatorname{Ker}(\psi_i))$ holds.
\end{cor}

\begin{proof}
The first assertion follows Lemmas~\ref{lem13} and \ref{lem15}.
The equality $\boldsymbol{E}(\varnothing) = \overline{\boldsymbol{T}(\boldsymbol{X})}^2$ implies that $\boldsymbol{V}(\operatorname{Ker}(\psi_i))$ is nonempty since $S_i$ is a semifield over $\boldsymbol{T}$ for each $i = 1, 2$.
Hence the conclusion follows Corollary~\ref{cor:connected}.
\end{proof}

Corollary~\ref{cor9} and \cite[Corollary~3.5 and Theorem~1.1]{JuAe7} induce the following corollary:

\begin{cor}
	\label{cor10}
In the above setting, if $\operatorname{Ker}(\psi)$ is finitely generated as a congruence, then so are both $\boldsymbol{E}(\boldsymbol{V}(\operatorname{Ker}(\psi_1)))$ and $\boldsymbol{E}(\boldsymbol{V}(\operatorname{Ker}(\psi_2)))$.
In particular, when $\operatorname{Ker}(\psi) = \boldsymbol{E}(\boldsymbol{V}(\operatorname{Ker}(\psi)))$ holds, $\operatorname{Ker}(\psi)$ is finitely generated as a congruence if and only if so are both $\operatorname{Ker}(\psi_1)$ and $\operatorname{Ker}(\psi_2)$.
\end{cor}

\begin{cor}
	\label{cor11}
Let $T$ be a finitely generated semifield over $\boldsymbol{T}$ and $\varphi \colon \overline{\boldsymbol{T}(Y_1, \ldots, Y_m)} \twoheadrightarrow T$ a surjective $\boldsymbol{T}$-algebra homomorphism.
If $\operatorname{Ker}(\varphi) = \boldsymbol{E}(\boldsymbol{V}(\operatorname{Ker}(\varphi)))$ and it is finitely generated as a congruence, then there exist no infinitely many semifields $T_1, T_2, \ldots$ over $\boldsymbol{T}$ such that $T$ is isomorphic to $\bowtie_{i = 1}^{\infty} T_i$ as a $\boldsymbol{T}$-algebra with the diagonal semiring homomorphism $\boldsymbol{T} \hookrightarrow \bowtie_{i = 1}^{\infty} T_i; t \mapsto (t, t, \ldots)$.
\end{cor}

\begin{proof}
If there exist infinitely many semifields $T_1, T_2, \ldots, $ over $\boldsymbol{T}$ such that $T \cong \bowtie_{i = 1}^{\infty} T_i$, then there exists a natural $\boldsymbol{T}$-algebra isomorphism $T_1 \bowtie \cdots \bowtie T_i \bowtie \left( \bowtie_{j = i + 1}^{\infty} T_j \right) \cong T_1 \bowtie \cdots \otimes T_{i + 1} \bowtie \left(\bowtie_{j = i + 2}^{\infty} T_j \right)$ fo any $i \ge 1$.
By applying Corollary~\ref{cor9} recursively, the equality $\boldsymbol{E}(\varnothing) = \overline{\boldsymbol{T}(Y_1, \ldots, Y_m)}^2$ and Corollary~\ref{cor:connected} ensure that all $T_j$s define nonempty disjoint congruence varieties whose union is $\boldsymbol{V}(\operatorname{Ker}(\varphi))$.
On the other hand, by \cite[Theorem~1.1]{JuAe7}, it cannot occur.
\end{proof}

\begin{lemma}
    \label{lem16}
Let $T$ be a finitely generated semifield over $\boldsymbol{T}$ and $\varphi$ a surjective $\boldsymbol{T}$-algebra homomorphism from a tropical rational function semifield $\overline{\boldsymbol{T}(Y_1, \ldots, Y_m)}  = \overline{\boldsymbol{T}(\boldsymbol{Y})}$ to $T$.
If $\operatorname{Ker}(\varphi) = \boldsymbol{E}(\boldsymbol{V}(\operatorname{Ker}(\varphi)))$ and $\boldsymbol{V}(\operatorname{Ker}(\varphi))$ is the disjoint union of two nonempty closed subsets $V_1$ and $V_2$ of $\boldsymbol{R}^m$, then $T$ is embedded into $\bowtie_{i = 1}^2 \overline{\boldsymbol{T}(\boldsymbol{Y})} / \boldsymbol{E}(V_i)$ as a $\boldsymbol{T}$-algebra with the diagonal semiring homomorphism $\boldsymbol{T} \hookrightarrow \bowtie_{i = 1}^2 \overline{\boldsymbol{T}(\boldsymbol{Y})} / \boldsymbol{E}(V_i); t \mapsto (t, t)$.
In addition, if $V_1$ and $V_2$ are finite unions of $\boldsymbol{R}$-rational polyhedral sets satisfying condition $(\ast)$ below, $T$ is isomorphic to $\bowtie_{i = 1}^2 \overline{\boldsymbol{T}(\boldsymbol{Y})} / \boldsymbol{E}(V_i)$ as a $\boldsymbol{T}$-algebra with the diagonal semiring homomorphism.

$(\ast)$ For $(i, j) = (1, 2), (2, 1)$, if $e$ is a ray of $V_i$, then there exist no vectors $\boldsymbol{a} \in \boldsymbol{R}^m$ such that $\{ \boldsymbol{x} + \boldsymbol{a} \,|\, \boldsymbol{x} \in e \}$ is a ray of $V_j$.
\end{lemma}

\begin{proof}
The equality $\boldsymbol{V}(\operatorname{Ker}(\varphi)) = \bigsqcup_{i = 1}^2 V_i$ implies $\boldsymbol{E}(\boldsymbol{V}(\operatorname{Ker}(\varphi))) \subset \boldsymbol{E}(V_i)$ for $i = 1, 2$.
As $\operatorname{Ker}(\varphi) = \boldsymbol{E}(\boldsymbol{V}(\operatorname{Ker}(\varphi)))$, there exists a surjective $\boldsymbol{T}$-algebra homomorphism $\gamma_i \colon \overline{\boldsymbol{T}(\boldsymbol{Y})} / \operatorname{Ker}(\varphi) \twoheadrightarrow \overline{\boldsymbol{T}(\boldsymbol{Y})} / \boldsymbol{E}(V_i)$ such that $\gamma_i \circ \pi_{\operatorname{Ker}(\varphi)} = \pi_{\boldsymbol{E}(V_i)}$ holds, where $\pi_{\operatorname{Ker}(\varphi)} \colon \overline{\boldsymbol{T}(\boldsymbol{Y})} \twoheadrightarrow \overline{\boldsymbol{T}(\boldsymbol{Y})} / \operatorname{Ker}(\varphi)$ and $\pi_{\boldsymbol{E}(V_i)} \colon \overline{\boldsymbol{T}(\boldsymbol{Y})} \twoheadrightarrow \overline{\boldsymbol{T}(\boldsymbol{Y})} / \boldsymbol{E}(V_i)$ stand for the natural $\boldsymbol{T}$-algebra homomorphisms, respectively.
Let $\gamma \colon \overline{\boldsymbol{T}(\boldsymbol{Y})} / \operatorname{Ker}(\varphi) \to \bowtie_{i = 1}^2 \overline{\boldsymbol{T}(\boldsymbol{Y})} / \boldsymbol{E}(V_i)$ be the map defined by the correspondence $\pi_{\operatorname{Ker}(\varphi)}(f) \mapsto (\gamma_1(\pi_{\operatorname{Ker}(\varphi)}(f)), \gamma_2(\pi_{\operatorname{Ker}(\varphi)}(f)))$ for any $f \in \overline{\boldsymbol{T}(\boldsymbol{Y})}$.
Since both $\gamma_1$ and $\gamma_2$ are surjective $\boldsymbol{T}$-algebra homomorphisms, so is $\gamma$.
For any $(f, g) \in \overline{\boldsymbol{T}(\boldsymbol{Y})}^2$ such that $\gamma(\pi_{\operatorname{Ker}(\varphi)}(f)) = \gamma(\pi_{\operatorname{Ker}(\varphi)}(g))$, by definition, $\gamma_i(\pi_{\operatorname{Ker}(\varphi)}(f)) = \gamma_i(\pi_{\operatorname{Ker}(\varphi)}(g))$ holds for $i = 1, 2$.
Hence for any $x \in \bigsqcup_{i = 1} V_i = \boldsymbol{V}(\operatorname{Ker}(\varphi))$, the equality $f(x) = g(x)$ holds.
This means that $(f, g) \in \boldsymbol{E}(\boldsymbol{V}(\operatorname{Ker}(\varphi))) = \operatorname{Ker}(\varphi)$ and $\gamma$ is injective.

Assume that $V_1$ and $V_2$ are finite unions of $\boldsymbol{R}$-rational polyhedral sets satisfying condition $(\ast)$.
By \cite[Lemmas~3.9, 3.10]{JuAe7}, for $i = 1, 2$, there exists $f_i \in \overline{\boldsymbol{T}(\boldsymbol{Y})}$ such that $\boldsymbol{V}((f_i, 0)) = V_i$ and $f_i(x) \ge \operatorname{dist}(x, V_i)$ hold for any $x \in \boldsymbol{R}^m$, where $\operatorname{dist}(x, V_i)$ denotes the distance between $x$ and $V_i$.
Thus by condition $(\ast)$, for any $g \in \overline{\boldsymbol{T}(\boldsymbol{Y})} \setminus \{ - \infty \}$, there exists a positive integer $k$ satisfying $g(x) \odot f_1(x)^{\odot (-k)} \le 0$ for any $x \in V_2$.
This means that
\begin{align*}
&~\pi_{\boldsymbol{E}(V_1)}\left( g \odot f_1^{\odot (-k)} \oplus 0 \right)\\
=&~\pi_{\boldsymbol{E}(V_1)}(g) \odot \pi_{\boldsymbol{E}(V_1)}(f_1)^{\odot (-k)} \oplus 0\\
=&~\pi_{\boldsymbol{E}(V_1)}(g) \oplus 0
\end{align*}
and
\begin{align*}
&~\pi_{\boldsymbol{E}(V_2)}\left( g \odot f_1^{\odot (-k)} \oplus 0 \right)\\
=&~\pi_{\boldsymbol{E}(V_2)}(g) \odot \pi_{\boldsymbol{E}(V_2)}(f_1)^{\odot (-k)} \oplus 0\\
=&~0
\end{align*}
hold.
Therefore $\gamma \left( g \odot f_1^{\odot (-k)} \oplus 0 \right)$ is $\left( \pi_{\boldsymbol{E}(V_1)}(g) \oplus 0, 0 \right)$.
By the same argument, it is verified that $\left( \pi_{\boldsymbol{E}(V_1)}(g)^{\odot (-1)} \oplus 0, 0 \right)$, and hence $\left( \pi_{\boldsymbol{E}(V_1)}(g), 0 \right) = \left( \left( \pi_{\boldsymbol{E}(V_1)}(g) \oplus 0 \right) \odot \left( \pi_{\boldsymbol{E}(V_1)}(g)^{\odot (-1)} \oplus 0 \right), 0 \right)$ are in the image of $\gamma$.
The same thing is true for $\left( 0, \pi_{\boldsymbol{E}(V_2)}(g) \right)$.
Since both $\gamma_1$ and $\gamma_2$ are surjective, so is $\gamma$.
\end{proof}

By Corollary~\ref{cor9} and Lemma~\ref{lem16}, condition $(3)$ of Corollary~\ref{cor:curve} is replaced by the condition that there exist no two semifields $S_1$ and $S_2$ over $\boldsymbol{T}$ such that $S$ is isomorphic to $S_1 \bowtie S_2$ as a $\boldsymbol{T}$-algebra with $d$.

Finally, we consider condition $(5)$ of Corollary~\ref{cor:curve}.
To do so, we introduce an invariant:

\begin{dfn}
    \label{dfn:e(S)}
\upshape{
Let $S$ be a finitely generated semifield over $\boldsymbol{T}$ and $\psi : \overline{\boldsymbol{T}(X_1, \ldots, X_n)} \twoheadrightarrow S$ a surjective $\boldsymbol{T}$-algebra homomorphism.
Assume that $\boldsymbol{V}(\operatorname{Ker}(\psi))$ is a finite union of $\boldsymbol{R}$-rational polyhedral sets.
Note here that by \cite[Corollary~3.20]{JuAe6}, if $\psi^{\prime} : \overline{\boldsymbol{T}(Y_1, \ldots, Y_m)} \twoheadrightarrow S$ is another surjective $\boldsymbol{T}$-algebra homomorphism, then $\boldsymbol{V}(\operatorname{Ker}(\psi^{\prime}))$ is also a finite union of $\boldsymbol{R}$-rational polyhedral sets.
We define $e(\psi)$ as the minimum of the numbers of unbounded cells of all $\boldsymbol{R}$-rational polyhedral structures on $\boldsymbol{V}(\operatorname{Ker}(\psi))$, and $e(S)$ as the minimum of $e(\psi)$ for such any $\psi$.
The number $e(S)$ depends only on the $\boldsymbol{T}$-algebra structure of $S$.
}
\end{dfn}

As in the setting of Definition~\ref{dfn:e(S)}, if $\operatorname{Ker}(\psi) = \boldsymbol{E}(\boldsymbol{V}(\operatorname{Ker}(\psi)))$ holds, then since $\boldsymbol{V}(\operatorname{Ker}(\psi))$ is a finite union of $\boldsymbol{R}$-rational polyhedral sets, there exist $f, g \in \overline{\boldsymbol{T}[X_1, \ldots, X_n]}$ such that $\langle (f, g) \rangle_{\overline{\boldsymbol{T}(X_1, \ldots, X_n)}} = \operatorname{Ker}(\psi)$ by \cite[Theorem~1.1 and Proposition~3.4]{JuAe7}.
Let $\overline{\boldsymbol{T}(\boldsymbol{X})}$ stand for $\overline{\boldsymbol{T}(X_1, \ldots, X_n)}$.
By Lemma~\ref{lem2}, the equality $\langle (f_{\boldsymbol{B}}, g_{\boldsymbol{B}}) \rangle_{\overline{\boldsymbol{T}(\boldsymbol{X})}} = \boldsymbol{E}(\boldsymbol{V}((f_{\boldsymbol{B}}, g_{\boldsymbol{B}})))$ holds. 
By Remark~\ref{rem5}, if $\langle (f^{\prime}, g^{\prime}) \rangle_{\overline{\boldsymbol{T}(\boldsymbol{X})}} = \operatorname{Ker}(\psi)$ holds, then $\boldsymbol{V}((f_{\boldsymbol{B}}, g_{\boldsymbol{B}})) = \boldsymbol{V}((f^{\prime}_{\boldsymbol{B}}, g^{\prime}_{\boldsymbol{B}}))$ holds, and thus, the congruence $\langle (f^{\prime}_{\boldsymbol{B}}, g^{\prime}_{\boldsymbol{B}} ) \rangle_{\overline{\boldsymbol{T}(\boldsymbol{X})}}$ coincides with $\boldsymbol{E}(\boldsymbol{V}((f_{\boldsymbol{B}}, g_{\boldsymbol{B}})))$.
Let $\psi_{\boldsymbol{B}}$ denote the natural surjective $\boldsymbol{T}$-algebra homomorphism $\overline{\boldsymbol{T}(\boldsymbol{X})} \twoheadrightarrow \overline{\boldsymbol{T}(\boldsymbol{X})} / \langle (f_{\boldsymbol{B}}, g_{\boldsymbol{B}}) \rangle_{\overline{\boldsymbol{T}(\boldsymbol{X})}}$ and $S_{\boldsymbol{B}}$ (one semifield over $\boldsymbol{T}$ isomorphic to) the quotient semifield $\overline{\boldsymbol{T}(\boldsymbol{X})} / \langle (f_{\boldsymbol{B}}, g_{\boldsymbol{B}}) \rangle_{\overline{\boldsymbol{T}(\boldsymbol{X})}}$ (as a $\boldsymbol{T}$-algebra).
Note that $e(S_{\boldsymbol{B}})$ is also defined for this $S_{\boldsymbol{B}}$ and by Lemma~\ref{lem6} and definition, $e(\psi) \ge e(\psi_{\boldsymbol{B}})$ and $e(S) \ge e(S_{\boldsymbol{B}})$ hold.
By \cite[Proposition~2.4.4]{Giansiracusa=Giansiracusa2}, the composition of $\psi_{\boldsymbol{B}}$ and any $\boldsymbol{T}$-algebra isomorphism $\phi : \overline{\boldsymbol{T}(\boldsymbol{X})} / \langle (f_{\boldsymbol{B}}, g_{\boldsymbol{B}}) \rangle_{\overline{\boldsymbol{T}(\boldsymbol{X})}} \to S_{\boldsymbol{B}}$ has $\operatorname{Ker}(\psi_{\boldsymbol{B}})$ as its kernel congruence.
Thus we have $e(\phi \circ \psi_{\boldsymbol{B}}) = e(\psi_{\boldsymbol{B}})$.

\begin{prop}
	\label{prop5}
In the above setting, additionally if $\operatorname{dim}S$ is one or two and $\boldsymbol{V}(\operatorname{Ker}(\psi))$ is connected, then the following are equivalent:

$(1)$ the equality $e(S) = e(S_{\boldsymbol{B}})$ holds,

$(2)$ the quotient semifield $\overline{\boldsymbol{T}(\boldsymbol{X})} / \operatorname{Ker}(\psi)$ and $\operatorname{Rat}(\boldsymbol{V}(\operatorname{Ker}(\psi))^{\prime})$ are isomorphic as $\boldsymbol{T}$-algebras, where $\boldsymbol{V}(\operatorname{Ker}(\psi))^{\prime}$ stands for the tropical curve obtained from $\boldsymbol{V}(\operatorname{Ker}(\psi))$ by one-point compactifications of rays of $\boldsymbol{V}(\operatorname{Ker}(\psi))$ with lattice length, and

$(3)$ the congruence variety $\boldsymbol{V}(\operatorname{Ker}(\psi))$ has no distinct rays whose primitive direction vectors coincide.
\end{prop}

\begin{proof}
If $\boldsymbol{V}(\operatorname{Ker}(\psi))$ has no rays, then the assertion is clear by \cite[Corollary~3.20]{JuAe6}.
Assume that $\boldsymbol{V}(\operatorname{Ker}(\psi))$ has at least one ray.
By \cite[Proposition~3.18]{JuAe7}, $(2)$ and $(3)$ are equivalent to each other.
Since $e(S)$ (resp.~$e(S_{\boldsymbol{B}})$) equals the number of rays of $\boldsymbol{V}((f, g))$ (resp.~$\boldsymbol{V}((f_{\boldsymbol{B}}, g_{\boldsymbol{B}}))$) by Lemma~\ref{lem6}, if $(3)$ holds, then $(1)$ holds, and vice versa.
\end{proof}

In conclusion, condition $(5)$ of Corollary~\ref{cor:curve} is equivalent to condition $(1)$ of Proposition~\ref{prop5}.

\end{document}